%% file: ex_article.tex
\documentclass[review,onefignum,onetabnum]{siamart220329}

\input{ex_shared}

\ifpdf
\hypersetup{
  pdftitle={Parameter-robust preconditioners},
  pdfauthor={X. Wang, X. Hu, and W. M. Boon}
}
\fi

\usepackage{tikz}
\usepackage{caption}
\usepackage{subcaption}
\usepackage{multirow}
\usepackage{diagbox}
\usepackage{bm}
\usepackage{amssymb}
\nolinenumbers 

\definecolor{darkgreen}{rgb}{0, .6, 0}
\newtheorem{assumption}{Assumption}[section]
\newtheorem{example}{Example}[section]

\newcommand{\enorm}[1]{\vert\kern-0.25ex \vert\kern-0.25ex \vert #1 \vert\kern-0.25ex \vert\kern-0.25ex \vert}

\begin{document}
\maketitle
\begin{abstract}
We consider the Stokes-Darcy coupled problem, which models the interaction between free-flow and porous medium flow.   By enforcing the normal flux continuity interface condition directly within the finite-element spaces, we establish unified well-posedness results for the coupled system under various boundary condition scenarios.  Using the operator preconditioning framework, we develop a parameter-robust preconditioner that avoids the use of fractional operators.  Numerical experiments employing both $H(\operatorname{div})$-conforming and nonconforming finite-element methods are presented to confirm the theoretical findings and demonstrate the robustness of the proposed block preconditioners with respect to the physical parameters and mesh size.
\end{abstract}



\section{Introduction}
The coupled interaction between free flow and porous media flow remains an active research area due to its wide applications. Examples include biomedical applications such as blood filtration \cite{rohan2021multiscale}, environmental processes like soil drying \cite{discacciati2009navier}, and engineering applications such as karst aquifers \cite{cao2010coupled} and filters design \cite{ervin2009coupled}.  In this work, we focus the simplified setting where the free-flow domain is governed by the stationary Stokes equations, while the porous medium is described by Darcy's law. 

There is a rich body of literature on the mathematical analysis, numerical discretizations, and iterative solution techniques for the coupled Stokes-Darcy model. Regarding numerical solutions, various discretizations have been studied, including (mixed) finite-element methods \cite{Rui2009A,bernardi2008mortar,Lv2024A}, finite difference/volume methods \cite{luo2017uzawa,rui2020mac,schneider2020coupling}, and other advanced numerical approaches \cite{wang2014spectral,wang2019divergence,zhao2020DG}. Beyond solving all variables simultaneously, domain decomposition methods and subdomain iteration methods have also been investigated in \cite{vassilev2014domain,discacciati2018optimized,galvis2007balancing,boon2024mortar}, which explore parallel and optimized iterative strategies for the coupled problem.

The discretization of the Stokes-Darcy model leads to large, sparse linear systems, for which Krylov iterative methods are typically applied to achieve efficient solutions. For the Darcy equations in primal form, a variety of solvers have been proposed. 
In \cite{cai2009preconditioning}, the authors constructed both block diagonal and block triangular preconditioners, while \cite{chidyagwai2016constraint} introduced an indefinite constraint preconditioner. In \cite{beik2022preconditioning}, an augmented Lagrangian approach was proposed, along with spectral and a field-of-values analysis for the block triangular preconditioner from \cite{cai2009preconditioning} and the constraint preconditioner from \cite{chidyagwai2016constraint}. The authors of \cite{greif2023block} used marker and cell (MAC) discretization and developed the block-structured preconditioner. In \cite{boon2022robust}, three different formulations were studied and the corresponding robust preconditioners were proposed based on properly chosen fractional Sobolev spaces. On the other hand, robust solvers for the Stokes–Darcy problem with mixed Darcy equations have also been developed. 
In \cite{luo2017uzawa}, an Uzawa smoother was designed for finite-volume discretization on staggered grids. In \cite{holter2020robust}, the authors enforced the flux continuity using Lagrange multipliers, formulating the problem in appropriately weighted fractional spaces and developing parameter-robust block preconditioners based on the operator preconditioning framework \cite{Kent2011Preconditioning}.  Additionally, \cite{boon2020parameter} introduced an interface flux variable and reduced the system to a problem involving only the interface flux. The corresponding preconditioners leveraged fractional operators, ensuring parameter robustness. 

To the best of our knowledge, all existing parameter-robust preconditioners for the coupled Stokes-Darcy problems rely on some type of fractional operators due to the interface conditions. In this paper, we aim to construct a mesh-independent and parameter-robust monolithic solver for the mixed formulation of the coupled Stokes-Darcy problem that avoids fractional operators entirely. We achieve this by imposing the mass conservation condition on the interface directly within the finite-element space, eliminating the need for Lagrange multipliers (or any type of interface variables).  We consider two different discretizations: an $H(\operatorname{div})$-conforming and a nonconforming finite-element method. Using an appropriately weighted norm, we prove the well-posedness of the discrete system and demonstrate that the associated constants are robust with respect to the physical parameters in most cases. This allows us to design a parameter-robust preconditioner within the operator preconditioning framework \cite{Kent2011Preconditioning}. Special attention is required for the cases where essential boundary conditions are imposed on one of the subproblems.  For these cases, we show that the inf-sup constants depend on the ratio between the material parameters in the two flow regimes but remain independent of the mesh size.  By considering proper quotient spaces, we prove that the inf-sup constant becomes parameter-robust again, thereby maintaining robustness of the preconditioner in terms of the effective condition number.  Our numerical experiments confirm the theoretical results, illustrating the robustness of the proposed preconditioner across a range of parameter regimes.

The paper is organized as follows. In \Cref{sec:model}, we introduce the coupled Stokes-Darcy model, derive its weak formulation, and present the general discrete scheme. The well-posedness of the discrete linear system under various boundary condition cases is established in \Cref{sec:well-posedness}. In \Cref{sec-preconditioner}, we develop the block diagonal preconditioner and analyze its robustness. Finally, numerical experiments demonstrating the performance of the proposed methods are provided in \Cref{sec:numerical}.

\section{The Model, Weak Formulation, and Discrete Scheme} \label{sec:model}

\subsection{Model}
Consider flow in a domain $\Omega\in\mathbb{R}^d$, $d = 2,3$,  consisting of a Stokes subregion $\Omega_S$ and a porous subregion $\Omega_D$, sharing a common interface $\Gamma = \partial\Omega_S\cap\partial\Omega_D$. Let $\Gamma_i = \partial\Omega_i\backslash\Gamma, i = S,D$, be the outer boundary. Also, let $\mathbf{n}_i$ be the unit outer normal vectors of $\partial\Omega_i,i = S,D$, especially, $\mathbf{n}_S = -\mathbf{n}_D$ on $\Gamma$.
Furthermore, denote the velocity by $\mathbf{u} = (\mathbf{u}_S,\mathbf{u}_D)$, the pressure by $p = (p_S,p_D)$, and $\mathbf{f} = (\mathbf{f}_S,\mathbf{f}_D)$ describes body forces. 

The flow in $\Omega_S$ is governed by the time-independent Stokes equations, while the porous medium subregion $\Omega_D$ is described by the Darcy equations,
\begin{align} 
-\nabla\cdot \bm{\sigma}(\mathbf{u}_S,p_S)=\mathbf{f}_S, \quad & \quad 
\nabla \cdot \mathbf{u}_S = 0, \quad \text { in } \Omega_S, \label{Stokes-model} \\
 \mathbf{K}^{-1} \mathbf{u}_D+\nabla p_D = \mathbf{f}_D, \quad & \quad
\nabla \cdot \mathbf{u}_D = 0, \quad  \text { in } \Omega_D, \label{Darcy-model}
\end{align}
Here,  $\bm{\sigma}(\mathbf{u}_S,p_S) = 2 \mu \bm{\epsilon}\left(\mathbf{u}_S\right)- p_S \mathbf{I} $ is the stress tensor and $\bm\epsilon(\mathbf{u}_S)=\frac{1}{2}\left(\nabla \mathbf{u}_S+\nabla \mathbf{u}_S^T\right)$ is the strain tensor. $\mu>0$ is the fluid viscosity, 
$\mathbf{K}$ denotes the hydraulic conductivity of the porous medium and, in this paper, we consider a homogeneous, isotropic porous medium, i.e. $\mathbf{K} = K\mathbf{I}$, $K>0$. Consequently, we will refer to it simply as $K$ moving forward. 

On the interface $\Gamma$, the following conditions are imposed \cite{Beavers1967Boundary,Saffman1971On}:
\begin{subequations}\label{interface}
\begin{align}
\mathbf{u}_S \cdot \mathbf{n}_S+\mathbf{u}_D \cdot \mathbf{n}_D = 0, &\quad \text { on } \Gamma, \label{interface-1}\\
- \bm{\sigma}(\mathbf{u}_S,p_S)\mathbf{n}_S\cdot \mathbf{n}_S = p_D,  &\quad \text { on } \Gamma, \label{interface-2} \\
- \bm{\sigma}(\mathbf{u}_S,p_S)\mathbf{n}_S \cdot \bm{\tau} = \beta_\tau\mathbf{u}_S \cdot \bm{\tau}, &\quad \text { on } \Gamma . \label{interface-3}
\end{align}
\end{subequations}
Here, \eqref{interface-1} represents the mass conservation on the interface, \eqref{interface-2} is the balance of normal forces, and \eqref{interface-3} denotes the Beavers-Joseph-Saffman (BJS) condition, in which $\beta_\tau = \alpha_{BJS}\frac{\mu}{\sqrt{\bm{\tau}\cdot\kappa\cdot\bm{\tau}}}$,  where $\kappa = \mu K$ is the permeability, with $\bm{\tau}$ denotes an orthonormal system of tangent vectors on $\Gamma$, and $\alpha_{BJS}\approx 1$ is an experimental parameter. Note that the BJS condition is suitable for flows that are predominantly tangential to the interface, and alternatives for more general settings have been proposed \cite{eggenweiler2020unsuitability,strohbeck2023modification}.

To close the coupled problem \eqref{Stokes-model}-\eqref{interface}, we describe the following homogeneous boundary conditions,
\begin{subequations}\label{BC}
\begin{align}
  \mathbf{u}_S = 0, \ \ & \text { on } \Gamma_S^{E}, & \bm{\sigma}(\mathbf{u}_S,p_S)\mathbf{n}_S = 0, & \ \  \text { on } \Gamma_S^{N}, \label{BC-1}\\
  \mathbf{u}_D \cdot \mathbf{n}_D=0, \ \ & \text { on } \Gamma_D^{E}, & p_D = 0, & \ \ \text { on } \Gamma_D^{N}.   \label{BC-2}
\end{align}
\end{subequations}
Here, the superscript distinguishes the essential $(E)$ and natural $(N)$ boundary conditions, and denote $\Gamma_S = \Gamma_S^{E}\cup\Gamma_S^{N}$ and $\Gamma_D = \Gamma_D^{E}\cup\Gamma_D^{N}$. In order to ensure uniqueness of the Stokes velocity, assume that $\Gamma_S^E \neq \emptyset$.
We comment that non-homogeneous or other types of boundary conditions can be considered as well, and our proposed method and analysis can be extended accordingly in a straightforward manner.

\subsection{Weak Formulation}
In this subsection, the weak formulation of the coupled Stokes-Darcy problem \cref{Stokes-model} will be presented. 
Firstly, we introduce some notations. Given a polygon $T$, the Sobolev space $H^m(T) = W^{2,m}(T)$ is defined in the usual way with norm $\|\cdot\|_{m,T}$. In particular, $H^{0}(T) = L^2(T)$ and we simply denote the $L^2$ inner-product and norm by $\|\cdot\|_{T}$ and $(\cdot,\cdot)_{T}$. Briefly, the $L^2$ inner-product and $L^2$ norm on $\Omega_i$ are denoted by $(\cdot,\cdot)_i$ and  $\|\cdot\|_{i}$, and the $L^2$ norm on the whole domain by $\|\cdot\| = \|\cdot\|_{\Omega}$. Also, $(\cdot,\cdot)_{\Gamma}$ and $\|\cdot\|_{\Gamma}$ are the $L^2$ inner-product and norm on the interface.

To present the weak form of the coupled problem \eqref{Stokes-model}-\eqref{BC}, let us introduce the appropriate velocity and pressure spaces in each domain,
\begin{align*}
   & \mathbf{V}_S = \left\{\mathbf{v}_S\in[H^1(\Omega_S)]^d| \ \mathbf{v}_S = 0 \text{ on } \Gamma_S^{E} \right\}, \quad Q_S = L^2(\Omega_S), \\
   & \mathbf{V}_D = \left\{\mathbf{v}_D\in \mathbf{H}(\operatorname{div}; \Omega_D)| \ \mathbf{v}_D\cdot\mathbf{n}_D = 0 \text{ on } \Gamma_D^{E} \right\}, \quad Q_D = L^2(\Omega_D),
\end{align*}
where
$\mathbf{H}(\operatorname{div};\Omega) = \left\{\mathbf{v}\in[L^2(\Omega)]^d| \ \nabla\cdot\mathbf{v}\in L^2(\Omega)\right\}$.
Then, the velocity and pressure spaces in the whole domain are
\begin{align*}
    & \mathbf{V} = \left\{\mathbf{v} = (\mathbf{v}_S,\mathbf{v}_D)\in\mathbf{V}_S\times\mathbf{V}_D| ~\mathbf{v}_S\cdot\mathbf{n}_S + \mathbf{v}_D\cdot\mathbf{n}_D = 0 \text{ on } \Gamma\right\}, \\
    & Q = \left\{q = (q_S,q_D)\in Q_S\times Q_D\right\}.
\end{align*}
It should be noted that if each subproblem is equipped with essential boundary conditions, that is, $\Gamma_i^N = \emptyset$ for both $i = S,D$, then the pressure space is redefined as 
$$Q  / \mathbb{R} = \left\{q = (q_S,q_D)\in Q_S\times Q_D| \ \int_{\Omega}q \mathrm{~d}\mathbf{x} = 0\right\}.$$

The weak form of the problem \cref{Stokes-model}-\cref{BC} is, find $\mathbf{u}\in \mathbf{V}, \ p\in Q$ such that
\begin{subequations}\label{weak-form}
\begin{alignat}{2}
a(\mathbf{u},\mathbf{v}) + b(\mathbf{v},p) & = (\mathbf{f},\mathbf{v}),   && \quad \forall\, \mathbf{v}\in\mathbf{V}, \\
b(\mathbf{u},q) & = 0, && \quad  \forall\, q\in Q,
\end{alignat}
\end{subequations}
where the bilinear forms $a(\cdot,\cdot):\mathbf{V}\times\mathbf{V}\rightarrow\mathbb{R}$ and $b(\cdot,\cdot):\mathbf{V}\times Q\rightarrow\mathbb{R}$ and the right-hand side are defined as
\begin{align*}
a(\mathbf{u},\mathbf{v}) &= 2\mu(\bm{\epsilon}(\mathbf{u}_S), \bm{\epsilon}(\mathbf{v}_S))_S + {K}^{-1}(\mathbf{u}_D,\mathbf{v}_D)_{D} + \beta_\tau(\mathbf{u}_S\cdot\bm{\tau}, \mathbf{v}_S\cdot\bm{\tau})_{\Gamma}, \\
b(\mathbf{v},p) &= - (\nabla \cdot \mathbf{v}_S, p_S)_S - (\nabla\cdot\mathbf{v}_D, p_D)_{D}, \\
(\mathbf{f},\mathbf{v}) &= (\mathbf{f}_S, \mathbf{v}_S)_S +  (\mathbf{f}_D,\mathbf{v}_D)_{D}. 
\end{align*}

\subsection{Discrete Scheme}
In this subsection, we propose a general finite-element scheme and give two specific finite-element methods in \Cref{sec-conforming} and \Cref{sec-nonconforming} as our examples. Let us introduce discrete notations. Let $\mathcal{T}_h^i$ be a family of triangulations of $\Omega_i, i=S, D$ (that is, triangles for $d=2$ and tetrahedrons for $d=3$), and assume that they match on $\Gamma$. Therefore, $\mathcal{T}_h:=\mathcal{T}_h^S\cup \mathcal{T}_h^D$ is a partition of $\Omega$, the interface $\Gamma_h = \mathcal{T}_h^S\cap\mathcal{T}_h^D$. We denote the diameter of any $T \in \mathcal{T}_h$ by $h_T$, and $h = \max\limits_{T\in\mathcal{T}_h}h_T$. Furthermore, the set of all interior edges (faces) is represented by $\mathcal{E}_h^0$ and the set of boundary edges (faces) is $\mathcal{E}_h^{\partial}$, then $\mathcal{E}_h = \mathcal{E}_h^0\cup \mathcal{E}_h^{\partial}$.

With every edge (face) $e \in \mathcal{E}_h$, we associate a unit vector $\mathbf{n}_e$ such that $\mathbf{n}_e$ is orthogonal to $e$ and equals the unit outer normal to $\partial \Omega$ if $e \subset \partial \Omega$. For any $e\in\mathcal{E}_h$ and a given function $\phi$, the jump of $\phi$ across $e$ is defined by $[\![\phi]\!]_e = \phi^{+} - \phi^{-}$ for $e\in \mathcal{E}_h^{0}$ and $[\![\phi]\!]_e = \phi^{+}$ for $e\in\mathcal{E}_h^{\partial}$,  where $\phi^{\pm} = \lim\limits_{s\rightarrow 0}\phi(\mathbf{x}\mp s\mathbf{n}_e)$ with $\mathbf{x}\in e$.

With the partitioning, we associate two finite-dimensional spaces $\mathbf{V}_h$ and $Q_h$:
\begin{align*}
& \mathbf{V}_h := \left\{\mathbf{v} = (\mathbf{v}_S,\mathbf{v}_{D})\in\mathbf{V}_{S,h}\times \mathbf{V}_{D,h}| \ \int_e \mathbf{v}_S\cdot\mathbf{n}_S + \mathbf{v}_{D}\cdot\mathbf{n}_{D} = 0, \ \forall \,e \in \mathcal{E}_h(\Gamma) \right\}, \\
& Q_h := \left\{q = (q_S,q_{D}) \in Q_{S,h}\times Q_{D,h}\right\}.
\end{align*} 
We specify two choices of finite-element spaces later in this paper.
Consistent with the continuous case, if $\Gamma_i^N = \emptyset, i = S,D$, the pressure space is defined as,
\begin{align*}
Q_h / \mathbb{R} = \left\{q = (q_S,q_{D}) \in Q_{S,h}\times Q_{D,h}| \ \int_{\Omega} q\mathrm{~d}\mathbf{x} = 0\right\}.
\end{align*}

Now we are ready to describe the discrete scheme corresponding to the coupled model \eqref{Stokes-model}-\eqref{BC} as follows: Find $(\mathbf{u},p)\in\mathbf{V}_h\times Q_h$ such that
\begin{subequations}\label{finite-element-scheme}
\begin{alignat}{2}
a_h\left(\mathbf{u}, \mathbf{v}\right)+b_h\left(\mathbf{v}, p\right)& =l_h\left(\mathbf{f}, \mathbf{v}\right), &&\quad  \forall\,\mathbf{v} \in \mathbf{V}_h, \label{finite-element-scheme-1}\\
b_h\left(\mathbf{u}, q\right)& = 0, && \quad \forall\, q \in Q_h,\label{finite-element-scheme-2}
\end{alignat}
\end{subequations}
where $a_h(\cdot,\cdot)$ and $l_h(\cdot,\cdot)$ depend on the specific finite-element method, which will be defined later.  In general, we can split the bilinear forms $a_h$ and $b_h$ into two components, which are defined in the Stokes region and Darcy region, respectively, 
\begin{subequations}
\begin{align}
a_h(\mathbf{u},\mathbf{v}) & = \mu a_{S,h}(\mathbf{u}_S,\mathbf{v}_S) + {K}^{-1} a_{D,h}(\mathbf{u}_{D},\mathbf{v}_{D}),\label{ah-def} \\
b_h(\mathbf{v},p) & = b_{S,h}(\mathbf{v}_S,p_S) + b_{D,h}(\mathbf{v}_{D},p_{D}) \nonumber\\
& = -(\nabla_h\cdot\mathbf{v}_S, p_S)_S -(\nabla_h\cdot\mathbf{v}_D, p_D)_{D} \nonumber\\
& = -\sum_{T\in\mathcal{T}_h^S}(\nabla_h\cdot\mathbf{v}_S, p_S)_{T} -\sum_{T\in\mathcal{T}_h^D}(\nabla_h\cdot\mathbf{v}_D, p_D)_{T}. \label{bh-def}
\end{align}
\end{subequations}
Here, $\nabla_h\cdot$ is the discrete divergence operator defined by $(\nabla_h\cdot\mathbf{v})|_{T} = \nabla\cdot\mathbf{v}|_{T}, \ \forall\, T\in\mathcal{T}_h.
$
The discrete problem \eqref{finite-element-scheme} can be represented as $\mathcal{A}\mathbf{x} = \mathbf{b}$, where
   \begin{equation} \label{eq: saddlepoint A}
   \mathcal{A} = 
   \left(
   \begin{array}{cc}
   A_{\mathbf{u}} & B^{T}  \\
   B &         0       \\
    \end{array}
    \right),
    \end{equation} 
and the blocks correspond to the bilinear forms $a_h(\mathbf{u},\mathbf{v}) \mapsto A_{\mathbf{u}}$ and $b_h(\mathbf{u}, q)\mapsto B $.

In what follows, we discuss two different finite-element discretizations for the coupled Stokes-Darcy problem \eqref{Stokes-model}-\eqref{BC} as examples. One is the $H(\operatorname{div})$-conforming element from \cite{Lv2024A} based on the Raviart-Thomas element. The other uses the non-conforming Crouzeix-Raviart element \cite{Rui2009A}.

\subsubsection{\texorpdfstring{$H(\operatorname{div})$}--conforming finite-element method}\label{sec-conforming}

We start by defining the spaces for the Stokes subdomain. Let $P_m(T)$, $m\geq 0$, denote the space of polynomials of degree less than or equal to $m$ on element $T$. In the Stokes domain $\Omega_S$, define the finite-element spaces for velocity as $\mathbf{V}_{S,h} :=\mathbf{V}_{S,h}^l \oplus \mathbf{R} \mathbf{T}_{S,0}$, where
\begin{align*}
\mathbf{V}_{S,h}^{l} &=\left\{\mathbf{v}_S^l \in\left[H^1\left(\Omega_S\right)\right]^d| \ \left.\mathbf{v}_S^l\right|_T \in\left[P_1(T)\right]^d, ~ \left.\mathbf{v}_S^l\right|_{\Gamma_S^{E}} = \mathbf{0}, ~\forall\,T \in \mathcal{T}_h^S\right\}, \\
\mathbf{R T}_{S,0} &=\left\{\mathbf{v}_S^R \in \mathbf{H}\left(\operatorname{div} ; \Omega_S\right)| \ \left.\mathbf{v}_S^R\right|_T \in\left[P_0(T)\right]^d \oplus \mathbf{x} P_0(T), \right. \\
&\qquad\qquad\qquad\qquad\qquad\qquad\qquad\qquad 
 \left. \left.\mathbf{v}_S^R \cdot \mathbf{n}_S\right|_{\Gamma_S^{E}}=0, ~\forall\,T \in \mathcal{T}_h^S\right\} .
\end{align*}
It is easy to prove that $\mathbf{V}_{S,h}^{l} \cap \mathbf{R T}_{S,0}=\{\mathbf{0}\}$, see the \cite[Lemma 2.1]{Li2022A}. Therefore, any $\mathbf{v}_S \in \mathbf{V}_{S,h}$ can be uniquely decomposed as
$\mathbf{v}_S=\mathbf{v}_S^l+\mathbf{v}_S^R \in \mathbf{V}_{S,h}^l\oplus\mathbf{R T}_{S,0}$.
The discrete pressure space on $\Omega_S$ is defined as
\begin{align*}
& Q_{S,h}=\left\{q_S \in L^2\left(\Omega_S\right)| \ \left.q_S\right|_T \in P_0(T), ~\forall\,T \in \mathcal{T}_h^S\right\} .
\end{align*}
In summary, $\mathbf{V}_{S,h}$ is the piecewise linear space ($\mathbb{P}_1$) plus standard Raviart-Thomas-N\'{e}d\'{e}lec space ($\mathbb{RT}_0$), and $Q_h$ is the piecewise constant space ($\mathbb{P}_0$). 

In the Darcy domain $\Omega_D$, we use the following velocity and pressure spaces, 
\begin{align*}
\mathbf{V}_{D,h}
&= \left\{\mathbf{v}_{D} \in \mathbf{H}\left(\operatorname{div}; \Omega_D\right)|  \left.\mathbf{v}_{D}\right|_T \in\left[P_0(T)\right]^d\oplus \mathbf{x} P_0(T),\right.\\
&\qquad\qquad\qquad\qquad\qquad\qquad
\hskip -4pt \left. \left.\mathbf{v}_{D} \cdot \mathbf{n}_D\right|_{\Gamma_D^{E}}=0, ~\forall\,T \in \mathcal{T}_h^D\right\},\\
Q_{D,h}&=\left\{q_{D} \in L^2\left(\Omega_D\right)| \ \left.q_{D}\right|_T \in P_0(T),  ~\forall\,T \in \mathcal{T}_h^D\right\}.
\end{align*}
That is, the choice of Darcy velocity and pressure spaces is $\mathbb{RT}_0$-$\mathbb{P}_0$. In this setting, the discrete problem is the $\mathbb{P}_1\oplus\mathbb{RT}_0$-$\mathbb{P}_0$ discretization of the Stokes' equation with $\mathbb{RT}_0$-$\mathbb{P}_0$ for the Darcy equation. The interface condition, $\int_e \mathbf{v}_S \cdot \mathbf{n}_S + \mathbf{v}_D \cdot \mathbf{n}_D = 0$, $\forall \, e \in \mathcal{E}_h(\Gamma)$, is imposed on the finite-element spaces directly by representing the Darcy flux degrees of freedom by the Stokes velocity degrees of freedomes on the interface, as shown in \cref{fig-dofs-conforming}.

Then, the corresponding discrete bilinear forms $a_{S,h}(\cdot,\cdot)$ and $a_{D,h}(\cdot,\cdot)$ in \eqref{ah-def} and right hand side term should be defined as follows,
\begin{equation*}
\begin{aligned}
a_{S,h}(\mathbf{u}_S, \mathbf{v}_S) & = 2(\bm{\epsilon}(\mathbf{u}_S^l), \bm{\epsilon}(\mathbf{u}_S^l ))_S + \widetilde{\beta}_\tau(\mathbf{u}_S^l\cdot\bm{\tau}, \mathbf{v}_S^l\cdot\bm{\tau})_{\Gamma} \\
&\quad +2 \sum_{T \in \mathcal{T}_h^S} \sum_{e \in \mathcal{E}_h(\bar{\Omega}_S)} \alpha_T u_e v_e(\nabla \cdot \bm{\phi}_e, \nabla \cdot \bm{\phi}_e)_T \\
&\quad + 2(\bm{\epsilon}(\mathbf{u}_S^l) \mathbf{n}_S \cdot \mathbf{n}_S, \mathbf{v}_S^R \cdot \mathbf{n}_S)_{\Gamma} + 2\gamma(\bm{\epsilon}(\mathbf{u}_S^R) \mathbf{n}_S \cdot \mathbf{n}_S, \mathbf{v}_S^l \cdot \mathbf{n}_S)_{\Gamma} \\
&\quad + \frac{\gamma^{\prime}}{h}(\mathbf{u}_S^R\cdot \mathbf{n}_S, \mathbf{v}_S^R\cdot\mathbf{n}_S)_{\Gamma}, \\
a_{D,h}(\mathbf{u}_D, \mathbf{v}_D) &  = (\mathbf{u}_{D},\mathbf{v}_{D})_{D}, \\
l_h(\mathbf{f}, \mathbf{v}) & = (\mathbf{f}_S, \Pi_h^R \mathbf{v}_S^l)_S + (\mathbf{f}_S, \mathbf{v}_S^R)_S + (\mathbf{f}_D, \mathbf{v}_{D})_D.
\end{aligned}
\end{equation*}
Here, the parameter $\widetilde{\beta}_\tau = \frac{\alpha_{BJS}}{\sqrt{\bm{\tau}\cdot\kappa\cdot\bm{\tau}}} = \mu^{-1} \beta_\tau$,  $\bm{\phi}_e$ is the basis function in $\mathbf{RT}_{S,0}$, and $\bar{\Omega}_S = \Omega_S\cup \partial\Omega_S$. For the definition of $a_{S,h}(\cdot,\cdot)$, the first two terms are natural from the Stokes equation by using the integration by parts and the interface conditions \eqref{interface-2}-\cref{interface-3}. The third term enriches the $\mathbb{P}_1$ space with the $\mathbb{RT}_0$ component, where $\alpha_T>0$ is a positive constant. The fourth term arises from the interface condition, while the fifth term, with $\gamma=1$, is added to ensure symmetry. For the last term, we choose $\gamma^{\prime}$ that is bounded by a positive constant to ensure the coercivity, the details can be seen in \cite[Remark 3.1]{Lv2024A}. Note that \cite{Lv2024A} does not include the last term because the authors choose $\gamma = -1$, causing the fourth and fifth terms to cancel out during their theoretical analysis.  In contrast, we select $\gamma = 1$ to maintain symmetry and include the last term to correctly establish coercivity. For instance, in our numerical tests, we set $\alpha_T = 5$ and $\gamma{\prime} = 1$. Additionally, $\Pi_h^R$ denotes the Raviart-Thomas interpolation operator, which ensures pressure-robustness.

\subsubsection{Nonconforming finite-element method}\label{sec-nonconforming}
Define the nonconforming Crouzeix-Raviart piecewise linear finite-element space ($\mathbb{CR}_0$) as follows,
\begin{align*}
\mathbf{V}_{S,h} =  \Bigg\{ \mathbf{v}|_T \in[P_1(T)]^d, ~\forall\,T \in \mathcal{T}_h, & \int_e[\![\mathbf{v}]\!]_e \mathrm{~d} s=0, ~ \forall\,e \in \mathcal{E}_h(\Omega_S)\cup\mathcal{E}_h(\Gamma_S^{E})\Bigg\}, \\
\mathbf{V}_{D,h} =  \Bigg\{ \mathbf{v}|_T \in[P_1(T)]^d, ~\forall\,T \in \mathcal{T}_h, & \int_e[\![\mathbf{v}]\!]_e \mathrm{~d} s=0, ~ \forall\,e \in \mathcal{E}_h(\Omega_D), \\
& \int_e[\![\mathbf{v} \cdot \mathbf{n}_e]\!]_e \mathrm{~d} s=0, ~\forall\,e \in \mathcal{E}_h(\Gamma_D^E) \Bigg\},
\end{align*}
and the pressure space is still the piecewise constant function space ($\mathbb{P}_0$). 

The corresponding bilinear forms are defined as follows,
\begin{align*} 
a_{S,h}(\mathbf{u}_S,\mathbf{v}_S) & =  2\sum_{T\in\mathcal{T}_h^S}(\bm{\epsilon}(\mathbf{u}_S),\bm{\epsilon}(\mathbf{v}_S))_T + \widetilde{\beta}_{\tau}
(\mathbf{u}_S\cdot\bm{\tau} , \mathbf{v}_S\cdot\bm{\tau})_{\Gamma} \\
 &\quad + \sum_{e\in\mathcal{E}_h^0(\Omega_S)} \int_{e} \frac{1}{h_e} [\![\mathbf{u}_S]\!]_e ~ [\![\mathbf{v}_S]\!]_e \mathrm{~d}s, \\
 a_{D,h}(\mathbf{u}_D,\mathbf{v}_D) & = (\mathbf{u}_{D},\mathbf{v}_{D})_D + \sum_{e\in\mathcal{E}_h^0(\Omega_d)} \int_{e} \frac{1}{h_e} [\![\mathbf{u}_{D}\cdot\mathbf{n}_{e}]\!]_e ~[\![\mathbf{v}_{D}\cdot\mathbf{n}_{e}]\!]_e \mathrm{~d}s, \\
l_h(\mathbf{f}, \mathbf{v}) & = (\mathbf{f}_S, \Pi_h^R \mathbf{v}_S)_S  + (\mathbf{f}_D, \mathbf{v}_{D})_D,
\end{align*}
 where $h_e$ is the length of edge $e$. The last terms of $a_{S,h}(\cdot,\cdot)$ and $a_{D,h}(\cdot,\cdot)$ are the penalizing term to account for the lack of $H^1$ and $H(\operatorname{div})$-conformity, respectively. Similarly, $\Pi_h^R$ is the Raviart-Thomas interpolant operator. Again the interface condition is imposed directly on the finite-element space by representing the Darcy velocity degrees of freedom by the Stokes velocity degrees of freedom, as shown in \cref{fig-dofs-non}.

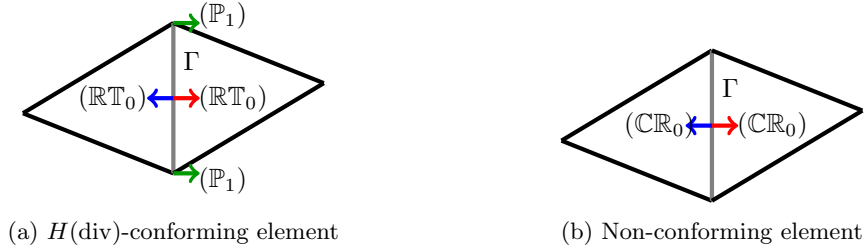
\begin{figure}
\centering
\begin{subfigure}{0.45\textwidth}
\centering
\begin{tikzpicture}[line width=1pt,scale=1]
\draw[black, ultra thick] (4,0.8) -- (6,2);
\draw[black, ultra thick] (4,0.8) -- (6,0);
\draw[gray, ultra thick] (6,2) -- (6,0);
\draw[black, ultra thick] (8,1.2) -- (6,2);
\draw[black, ultra thick] (8,1.2) -- (6,0);
\draw[darkgreen, ultra thick, ->] (6,0) -- (6.35,0);
\draw[darkgreen, ultra thick, ->] (6,2) -- (6.35,2);
\draw[red, ultra thick, ->] (6,1) -- (6.35,1);
\draw[blue, ultra thick, ->] (6,1) -- (5.65,1);
\filldraw[black] (6.2,2.1) node[anchor=west]{($\mathbb{P}_1$)};
\filldraw[black] (6.2,-0.1) node[anchor=west]{($\mathbb{P}_1$)};
\filldraw[black] (6.2,1.0) node[anchor=west]{($\mathbb{RT}_0$)};
\filldraw[black] (4.6,1.0) node[anchor=west]{($\mathbb{RT}_0$)};
\filldraw[black] (6,1.5) node[anchor=west]{$\Gamma$};
\end{tikzpicture}
\caption{$H(\operatorname{div})$-conforming element}
\label{fig-dofs-conforming}
\end{subfigure}
\hfill
\begin{subfigure}{0.45\textwidth}
\centering
\begin{tikzpicture}[line width=1pt,scale=1]
\draw[black, ultra thick] (4,0.8) -- (6,2);
\draw[black, ultra thick] (4,0.8) -- (6,0);
\draw[gray, ultra thick] (6,2) -- (6,0);
\draw[black, ultra thick] (8,1.2) -- (6,2);
\draw[black, ultra thick] (8,1.2) -- (6,0);
\draw[red, ultra thick, ->] (6,1) -- (6.35,1);
\draw[blue, ultra thick, ->] (6,1) -- (5.65,1);
\filldraw[black] (6.2,1.0) node[anchor=west]{($\mathbb{CR}_0$)};
\filldraw[black] (4.7,1.0) node[anchor=west]{($\mathbb{CR}_0$)};
\filldraw[black] (6,1.5) node[anchor=west]{$\Gamma$};
\end{tikzpicture}
\caption{Nonconforming element}
\label{fig-dofs-non}
\end{subfigure}
\caption{Two choices of finite-element spaces based on the Raviart-Thomas or the Crouzeix-Raviart element, respectively. The Darcy flux degrees of freedom (red) are determined by the Stokes velocity degrees of freedom (blue and green).}
\end{figure}

In the following section, we will prove that both schemes are well-posed under different choices of boundary conditions.  

\section{Well-posedness Analysis}\label{sec:well-posedness}
In this section, we focus on analyzing the well-posedness of linear system \eqref{finite-element-scheme} under an appropriate weighted norm. This analysis forms the foundation for designing a parameter-robust preconditioner using the operator preconditioning framework \cite{Kent2011Preconditioning}.

We define the following weighted norms for $(\mathbf{v}, q)\in\mathbf{V}_h \times Q_h$:
\begin{subequations} \label{eqs: norms v and Q}
\begin{align}
   & \|\mathbf{v}\|_{V_h}^2 =
   \mu\|\mathbf{v}_S\|_{V_{S,h}}^2 + {K}^{-1}\|\mathbf{v}_D\|_{V_{D,h}}^2 \nonumber \\
   & \qquad\qquad + \omega_S\mu\|P_{S,h}(\nabla_h\cdot\mathbf{v}_S)\|_S^2 + \omega_D{K}^{-1}\|\nabla_h\cdot\mathbf{v}_{D}\|_{D}^2, \label{u-norm}\\
   & \|q\|_{Q_h}^2 =  \omega_S^{-1}\mu^{-1}\| q_S\|_S^2 + \omega_D^{-1}{K}\|q_{D}\|_{D}^2,   \label{p-norm}
\end{align} 
\end{subequations}
where
$
 \|\mathbf{v}_S\|_{V_{S,h}}^2 =  a_{S,h}(\mathbf{v}_S,\mathbf{v}_S),\   
 \|\mathbf{v}_D\|_{V_{D,h}}^2 = a_{D,h}(\mathbf{v}_D,\mathbf{v}_D).
$
The parameters $\omega_S,\,\omega_D \geq 1$ are tunable and influence the eigenvalue distribution of the preconditioned system \cite{benzi2011analysis}. In addition, $P_{S,h}: L^2(\Omega_S)\rightarrow Q_{S,h}$ is a $L^2$-projection operator defined as
\begin{align*}
  & (P_{S,h}(\nabla_h\cdot\mathbf{v}_S), q_S)_S = (\nabla_h\cdot\mathbf{v}_S, q_S)_S, \ \ \forall \,q_S\in Q_{S,h},
\end{align*}
and there holds the estimate $\|P_{S,h}(\nabla\cdot\mathbf{v}_S)\|_S \leq  \|\nabla\cdot\mathbf{v}_S\|_S$.

Furthermore, \eqref{u-norm} and \eqref{p-norm} together yield the following triple norm,
\begin{align}\label{norm}
    \enorm{(\mathbf{v},q)}
     = & \big(\|\mathbf{v}\|_{V_h}^2 + \|q\|_{Q_h}^2\big)^{\frac12}.
\end{align}

\begin{remark}
    In \cref{u-norm}, the last term $\|\nabla_h\cdot\mathbf{v}_D\|_D$ is natural since $\mathbf{v}_D$ belongs to the space $\mathbf{H}(\operatorname{div};\Omega_D)$. On the other hand, the divergence term of Stokes velocity $\|P_{S,h}(\nabla_h\cdot\mathbf{v}_S)\|_S$ is related to the idea of augmented Lagrangian method \cite{fortin2000augmented}, and has been widely used in developing parameter-robust block preconditioner \cite{benzi2011analysis,moulin2019augmented,farrell2019augmented,beik2022preconditioning}.
\end{remark}

Besides properly weighted norms for the well-posedness of the Stokes-Darcy problem, it requires the Stokes inf-sup stability and the Darcy inf-sup stability, respectively. Therefore, we make the following assumptions. 
\begin{assumption}\label{ass:ass1} 
    The pair $\mathbf{V}_{S,h}\times Q_{S,h}$ is stable for the incompressible
    Stokes equations \eqref{Stokes-model}, i.e. the discrete Stokes inf-sup condition holds: There exists $\gamma_S>0$ independent of $h$, $\mu$ and $\beta_{\tau}$ such that,
\begin{align}\label{inf-sup-Stokes}
  \mathop{\sup}\limits_{\substack{\mathbf{v}_S\in\mathbf{V}_{S,h}\\ \mathbf{v}_S|_{\Gamma} = 0}} 
  \frac{(\nabla_h\cdot\mathbf{v}_S,q_S)}{\|\mathbf{v}_S\|_{V_{S,h}}} 
  &\geq \gamma_S \| q_S\|_S, 
  & \forall\, q_S&\in Q_{S,h}.
\end{align}
Moreover, if $|\Gamma_S^N| = 0$, then \eqref{inf-sup-Stokes} holds for $q_S \in Q_{S,h} / \mathbb{R}$.
\end{assumption}
\begin{assumption}\label{ass:ass2} The spaces $\mathbf{V}_{D,h}\times Q_{D,h}$ form a stable pair for the Darcy equations \eqref{Darcy-model}, i.e. there exists $\gamma_D>0$ independent of $h$ and $K$ such that,
\begin{align}\label{inf-sup-Darcy}
  \mathop{\sup}\limits_{\substack{\mathbf{v}_D\in\mathbf{V}_{D,h}\\ \mathbf{n} \cdot \mathbf{v}_D|_{\Gamma} = 0}} 
  \frac{(\nabla_h\cdot\mathbf{v}_D,q_D)}{(\|\mathbf{v}_D\|_{V_{D,h}}^2 + \| \nabla_h \cdot \mathbf{v}_D\|_D^2)^{\frac12}} 
  &\geq \gamma_D \| q_D\|_D, 
  & \forall\, q_D&\in Q_{D,h}.
\end{align}
Moreover, if $|\Gamma_D^N| = 0$, then \eqref{inf-sup-Darcy} holds for $q_D \in Q_{D,h} / \mathbb{R}$.
\end{assumption}

While each subproblem is inf-sup stable in terms of the physical parameters, the coupled Stokes-Darcy problem may be unstable under certain combinations of material parameters and boundary conditions.  We showcase this by presenting two numerical examples, \cref{ex3.1} and \cref{ex3.2}, where one of the subproblems is subjected to no-flux boundary conditions. For simplicity, we fix the parameters $\alpha_{BJS} = 1$, and use the nonconforming element $\mathbb{CR}_0$-$\mathbb{P}_0$-$\mathbb{CR}_0$-$\mathbb{P}_0$ for the discretization. The results are similar for $\mathbb{P}_1\oplus\mathbb{RT}_0$-$\mathbb{P}_0$-$\mathbb{RT}_0$-$\mathbb{P}_0$.
\begin{example}[Darcy subproblem with no-flux boundary conditions]\label{ex3.1}  
Consider the case in which Stokes subproblem has mixed boundary conditions, i.e. $\Gamma_S^N \neq \emptyset$, and the Darcy subproblem has pure Dirichlet boundary conditions, i.e. $\Gamma_D^N = \emptyset$.  
We present the inf-sup condition constant in \cref{table-Darcy-Dirichlet} for different $\mu K$. It is clear that the inf-sup condition constant is independent of the mesh size but depends on the material parameters. We can see that the larger the $\mu K$, the smaller the inf-sup constant. 
\end{example}

\begin{table}[htbp]
  \centering
  \caption{Inf-sup constant for Stokes subproblem with mixed boundary condition, Darcy subproblem with pure Dirichlet boundary condition.}
  \begin{tabular}{|c|ccccc|}
    \hline \hline 
    $\mu K \backslash h$ & 1/2  & 1/4 & 1/8 & 1/16 & 1/32 \\ 
    \hline
$1$   & 4.29e-01   & 4.21e-01   & 4.15e-01   & 4.12e-01   & 4.10e-01  \\
$10^{2}$ & 4.95e-02 & 4.84e-02   & 4.76e-02   & 4.72e-02   & 4.70e-02  \\
$10^{4}$ & 4.96e-03  & 4.86e-03   & 4.78e-03   & 4.74e-03   & 4.72e-03  \\
    \hline \hline 
  \end{tabular}
  \label{table-Darcy-Dirichlet}
\end{table}

\begin{example}[Stokes subproblem with velocity boundary conditions]
\label{ex3.2}
Consider Stokes subproblem with Dirichlet boundary condition, i.e. $\Gamma_S^N = \emptyset$, and Darcy subproblem with mixed boundary condition, i.e. $\Gamma_D = \Gamma_D^N\cup \Gamma_D^E$. \cref{table-Stokes-Dirichlet} shows the inf-sup constants for various values of $\mu K$, illustrating that as $\mu K$ decreases, the inf-sup constants become smaller.
\end{example}
\begin{table}[htbp]
  \centering
  \caption{Inf-sup constant for Stokes subproblem with Dirichlet boundary condition, Darcy subproblem with mixed boundary condition.}
  \begin{tabular}{|c|ccccc|}
    \hline \hline 
    $\mu K \backslash h$  & 1/2 & 1/4 & 1/8 & 1/16 & 1/32 \\ 
    \hline
$10^{-4}$ & 1.89e-01 & 2.08e-01   & 2.23e-01   & 2.34e-01   & 2.36e-01    \\
$10^{-2}$ & 8.64e-01  & 8.61e-01   & 8.52e-01   & 8.45e-01   & 8.42e-01 \\
$1$   & 9.69e-01    & 9.63e-01   & 9.60e-01   & 9.59e-01   & 9.58e-01 \\
    \hline \hline 
  \end{tabular}
  \label{table-Stokes-Dirichlet}
\end{table}

These two examples demonstrate that the configuration of the boundary conditions influences the inf-sup constants. Therefore, in what follows, we proceed to analyze the well-posedness of the Stokes-Darcy problem equipped with different boundary conditions. Given that the coupled system \eqref{finite-element-scheme} is a saddle-point system, based on the Brezzi theory \cite{boffi2013mixed}, we need to verify the continuity of the bilinear forms, coercivity of $a_h$ on the kernel of $b_h$, and the inf-sup condition on $b_h$. We start with the continuity and coercivity estimates in the following lemma.

\begin{lemma}\label{lemma-1}
The bilinear forms $a_h(\cdot,\cdot)$ and $b_h(\cdot,\cdot)$ satisfy 
\begin{subequations}\label{coer-continuity}
\begin{align}
& a_h(\mathbf{u}, \mathbf{v})  \leq\|\mathbf{u}\|_{V_h}\|\mathbf{v}\|_{V_h}, \quad \forall\,\mathbf{u}, \mathbf{v} \in \mathbf{V}_h, \label{a-continuity}\\
& b_h(\mathbf{v}, q)  \leq \|\mathbf{v}\|_{V_h} \| q \|_{Q_h}, 
\quad \forall\,(\mathbf{v}, q) \in \mathbf{V}_h \times Q_h, \label{b-continuity}\\
&a_h(\mathbf{v}, \mathbf{v}) \geq\|\mathbf{v}\|_{V_h}^2, 
\quad \forall\,\mathbf{v}\in\mathbf{V}_h \text{ with } b_h(\mathbf{v},q) = 0, ~\forall\,q\in Q_h. \label{a-coercivity} 
\end{align}
\end{subequations}
\end{lemma}
\begin{proof}
\eqref{a-continuity} and \eqref{b-continuity} follow from the Cauchy-Schwarz inequality, the definition of the bilinear forms and the corresponding norms.  For \eqref{a-coercivity}, if $b_h(\mathbf{v}, q)=0$ for all $q \in Q_h$, then $P_{S,h}(\nabla_h\cdot  \mathbf{v}_S) = 0$ and $\nabla_h\cdot\mathbf{v}_D = 0$. Thus, it is clear that the coercivity of $a_h(\cdot,\cdot)$ holds in the kernel space of $\mathbf{V}_h$.
\end{proof}

Note that the constants in \Cref{lemma-1} are independent of the material parameters, discretization parameters, and the chosen boundary conditions. The behavior observed in \cref{ex3.1} and \cref{ex3.2} is therefore due to the inf-sup condition on the bilinear form $b_h$. To address this, we distinguish the following four cases depending on whether natural boundary conditions are imposed or not, see \Cref{fig:bc_config}:
\begin{itemize}
   \item \textbf{Case NN}: In this case, $|\Gamma_i^N|>0$ for both $i \in \{S, D\}$.
   \item \textbf{Case EE}: The (normal) flux is imposed on the entire domain boundary, i.e. $|\Gamma_i^N| = 0$ for both $i \in \{S, D\}$.
   \item \textbf{Case NE}: Only the Stokes domain has natural (traction) boundary conditions, i.e. $|\Gamma_S^N|>0$ and $|\Gamma_D^N|=0$.
   \item \textbf{Case EN}: Only the Darcy domain has natural (pressure) boundary conditions, i.e. $|\Gamma_S^N|=0$ and $|\Gamma_D^N|>0$.
\end{itemize}

\begin{remark}
  If the porous medium $\Omega_D$ is entirely surrounded by the Stokes domain, then $|\Gamma_D^N| = 0$. Therefore, this configuration falls under cases \textbf{EE} or \textbf{NE}, depending on the boundary conditions imposed on $\partial \Omega = \Gamma_S$. Conversely, if the Stokes domain is entirely surrounded by the porous medium domain, then the setup corresponds to cases \textbf{EE} or \textbf{EN}, depending on the boundary conditions prescribed on $\partial \Omega = \Gamma_D$. 
\end{remark}

\begin{figure}[ht]
    \centering
\begin{subfigure}{0.45\textwidth}
    \centering
\begin{tikzpicture}[line width=1pt,scale=2]
\draw[gray, ultra thick] (1,0) -- (1,1);
\draw[black, ultra thick] (0,0) -- (2,0);
\draw[black, ultra thick] (0,0) -- (0,1);
\draw[black, ultra thick] (0,1) -- (2,1);
\draw[black, ultra thick] (2,0) -- (2,1);
\filldraw[black] (0.4,-0.15) node[anchor=west]{$\Gamma_S^N$};
\filldraw[black] (0.4,1.15) node[anchor=west]{$\Gamma_S^N$};
\filldraw[black] (-0.4,0.5) node[anchor=west]{$\Gamma_S^E$};
\filldraw[black] (1.4,-0.15) node[anchor=west]{$\Gamma_D^N$};
\filldraw[black] (1.4,1.15) node[anchor=west]{$\Gamma_D^N$};
\filldraw[black] (2.02,0.5) node[anchor=west]{$\Gamma_D^E$};
\filldraw[black] (0.4,0.5) node[anchor=west]{$\Omega_S$};
\filldraw[black] (1.4,0.5) node[anchor=west]{$\Omega_D$};
\end{tikzpicture}
\caption{\textbf{Case NN}}
\end{subfigure}
\hfill
\begin{subfigure}{0.45\textwidth}
    \centering
\begin{tikzpicture}[line width=1pt,scale=2]
\draw[gray, ultra thick] (1,0) -- (1,1);
\draw[black, ultra thick] (0,0) -- (2,0);
\draw[black, ultra thick] (0,0) -- (0,1);
\draw[black, ultra thick] (0,1) -- (2,1);
\draw[black, ultra thick] (2,0) -- (2,1);
\filldraw[black] (0.4,-0.15) node[anchor=west]{$\Gamma_S^E$};
\filldraw[black] (0.4,1.15) node[anchor=west]{$\Gamma_S^E$};
\filldraw[black] (-0.4,0.5) node[anchor=west]{$\Gamma_S^E$};
\filldraw[black] (1.4,-0.15) node[anchor=west]{$\Gamma_D^E$};
\filldraw[black] (1.4,1.15) node[anchor=west]{$\Gamma_D^E$};
\filldraw[black] (2.02,0.5) node[anchor=west]{$\Gamma_D^E$};
\filldraw[black] (0.4,0.5) node[anchor=west]{$\Omega_S$};
\filldraw[black] (1.4,0.5) node[anchor=west]{$\Omega_D$};
\end{tikzpicture}
\caption{\textbf{Case EE}}
\end{subfigure}
\begin{subfigure}{0.45\textwidth}
    \centering
\begin{tikzpicture}[line width=1pt,scale=2]
\draw[gray, ultra thick] (1,0) -- (1,1);
\draw[black, ultra thick] (0,0) -- (2,0);
\draw[black, ultra thick] (0,0) -- (0,1);
\draw[black, ultra thick] (0,1) -- (2,1);
\draw[black, ultra thick] (2,0) -- (2,1);
\filldraw[black] (0.4,-0.15) node[anchor=west]{$\Gamma_S^N$};
\filldraw[black] (0.4,1.15) node[anchor=west]{$\Gamma_S^N$};
\filldraw[black] (-0.4,0.5) node[anchor=west]{$\Gamma_S^E$};
\filldraw[black] (1.4,-0.15) node[anchor=west]{$\Gamma_D^E$};
\filldraw[black] (1.4,1.15) node[anchor=west]{$\Gamma_D^E$};
\filldraw[black] (2.02,0.5) node[anchor=west]{$\Gamma_D^E$};
\filldraw[black] (0.4,0.5) node[anchor=west]{$\Omega_S$};
\filldraw[black] (1.4,0.5) node[anchor=west]{$\Omega_D$};
\end{tikzpicture}
\caption{\textbf{Case NE}}
\end{subfigure}
\hfill
\begin{subfigure}{0.45\textwidth}
    \centering
\begin{tikzpicture}[line width=1pt,scale=2]
\draw[gray, ultra thick] (1,0) -- (1,1);
\draw[black, ultra thick] (0,0) -- (2,0);
\draw[black, ultra thick] (0,0) -- (0,1);
\draw[black, ultra thick] (0,1) -- (2,1);
\draw[black, ultra thick] (2,0) -- (2,1);
\filldraw[black] (0.4,-0.15) node[anchor=west]{$\Gamma_S^E$};
\filldraw[black] (0.4,1.15) node[anchor=west]{$\Gamma_S^E$};
\filldraw[black] (-0.4,0.5) node[anchor=west]{$\Gamma_S^E$};
\filldraw[black] (1.4,-0.15) node[anchor=west]{$\Gamma_D^N$};
\filldraw[black] (1.4,1.15) node[anchor=west]{$\Gamma_D^N$};
\filldraw[black] (2.02,0.5) node[anchor=west]{$\Gamma_D^E$};
\filldraw[black] (0.4,0.5) node[anchor=west]{$\Omega_S$};
\filldraw[black] (1.4,0.5) node[anchor=west]{$\Omega_D$};
\end{tikzpicture}
\caption{\textbf{Case EN}}
\end{subfigure}
\caption{The four cases with different boundary conditions.}
\label{fig:bc_config}
\end{figure}
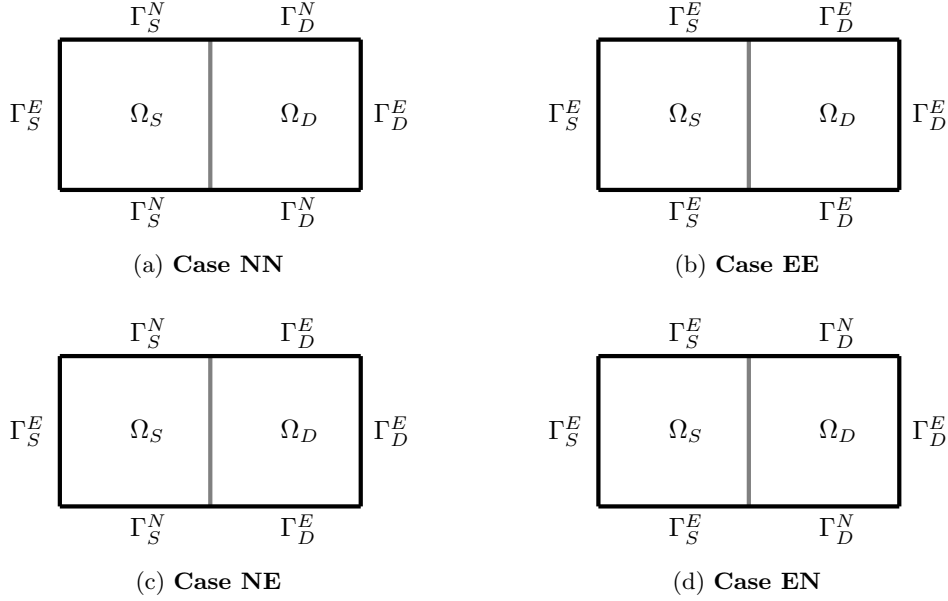

We continue by proving the inf-sup condition for these four cases, respectively, starting with the simplest case, i.e., \textbf{Case NN}.

\begin{lemma}[\textbf{Case NN}] \label{lemma: NN}
   Under \cref{ass:ass1} and \cref{ass:ass2}, there exists $\gamma > 0$ independent of material parameters and $h$ such that,
    \begin{align}
       &\inf_{q\in Q_h}\sup_{\mathbf{v}\in\mathbf{V}_h}\frac{b_h(\mathbf{v}, q) }{\|\mathbf{v}\|_{V_h}\|q\|_{Q_h}} \geq \gamma.
    \end{align}
\end{lemma}
\begin{proof}
    In this case, both the Stokes and Darcy subdomains have a boundary portion on which the stress, respectively pressure, is imposed. We may therefore use the inf-sup conditions from \Cref{ass:ass1} and \Cref{ass:ass2} to construct a suitable test function with zero normal flux across $\Gamma$. We present a concise analysis to verify that $\gamma$ is independent of material parameters.  

    Let $q=\left(q_S, q_D\right) \in Q_h$ be given. Let $\mathbf{v}_S^0$ and $\mathbf{v}_D^0$ be chosen such that
    \begin{align}
        \int_e \mathbf{n}_S \cdot \mathbf{v}_S^0 = \int_e \mathbf{n}_S \cdot \mathbf{v}_D^0 &= 0, & 
        \forall\, e \in \Gamma_h, \label{eq: zeroflux}
    \end{align}
    and moreover
    \begin{subequations} \label{eqs: properties v0}
    \begin{align}
        P_{S, h} \nabla_h \cdot \mathbf{v}_S^0 &= q_S, 
        & \|\mathbf{v}_S^0\|_{V_{S,h}}^2 + \|P_{S,h}(\nabla_h\cdot\mathbf{v}_S^0)\|_S^2 
        &\lesssim \| q_S \|_S^2, \\
        \nabla_h \cdot \mathbf{v}_D^0 &= q_D, &        
        \|\mathbf{v}_D^0\|_{V_{D,h}}^2 + \|\nabla_h\cdot\mathbf{v}_D^0\|_D^2
        &\lesssim \| q_D \|_D^2.
    \end{align}
    \end{subequations}
    
    Now define our test function $\mathbf{v} := -(\omega_S^{-1}\mu^{-1} \mathbf{v}_S^0, \omega_D^{-1}{K} \mathbf{v}_D^0)$ and note that it belongs to $\mathbf{V}_h$ due to the zero flux condition \eqref{eq: zeroflux}. The following bounds now follow:
    \begin{subequations} \label{eqs: infsup NN}
        \begin{align}        
            b_h(\mathbf{v}, q) &= \omega_S^{-1}\mu^{-1} \| q_S \|_S^2 + \omega_D^{-1}{K} \| q_D \|_D^2 = \|q\|_{Q_h}^2, \\
            \| \mathbf{v} \|_{V_h}^2 &= 
            \mu \|\omega_S^{-1} \mu^{-1} \mathbf{v}_S^0\|_{V_{S,h}}^2 + \omega_S\mu\| \omega_S^{-1}\mu^{-1}P_{S,h}(\nabla_h\cdot\mathbf{v}_S^0)\|_S^2 \nonumber \\
            &\quad + {K}^{-1}\| \omega_D^{-1}{K} \mathbf{v}_D^0\|_{V_{D,h}}^2 + \omega_D {K}^{-1}\|\omega_D^{-1} {K} \nabla_h\cdot\mathbf{v}_D^0\|_D^2 \nonumber \\
            &\lesssim  \omega_S^{-1}\mu^{-1} \| q_S \|_S^2 + \omega_D^{-1} {K} \| q_D \|_D^2  = \|q\|_{Q_h}^2, 
        \end{align}
    \end{subequations}
    where we used the fact that $\omega_S, \omega_D \ge 1$. Together, \eqref{eqs: infsup NN} proves the result.
\end{proof}

In the remaining three cases, at least one subproblem has essential boundary conditions imposed on the entire boundary. Consequently, the approach used in the previous lemma cannot be applied directly, as it does not provide control over the constants. To address this, we introduce a key tool in the following lemma.

\begin{lemma} \label{lem: construction}
    Given $\bar f \in \mathbb{R}$, there exists a $\mathbf{v}^\Gamma = (\mathbf{v}_S^\Gamma, \mathbf{v}_D^\Gamma) \in \mathbf{V}_h$ that satisfies
    \begin{align} \label{eq: property vGamma}
        (\nabla_h \cdot \mathbf{v}_S^\Gamma, 1)_S
        = -(\nabla_h \cdot \mathbf{v}_D^\Gamma, 1)_D
        = \bar f,
    \end{align}
    and is bounded as
    \begin{align} \label{eq: bound on vGamma}
        \| \mathbf{v}_S^\Gamma \|_{V_{S,h}}
        + \| \mathbf{v}_D^\Gamma \|_{V_{D,h}}
        + \| \nabla_h \cdot \mathbf{v}_S^\Gamma \|_S
        + \| \nabla_h \cdot \mathbf{v}_D^\Gamma \|_D
        \lesssim
        |\bar f|.
    \end{align}
\end{lemma}
\begin{proof}
    Similar to \cite[Lem.~3.1]{boon2020parameter}, we first prescribe the normal flux on the interface $\Gamma$ and then create bounded extensions into the two subdomains. 
    We start by introducing $\phi \in H_0^1(\Gamma)$ as the solution to the following minimization problem
    \begin{align}
        \min_{\phi \in H_0^1(\Gamma)} & \frac12 \| \nabla_\Gamma \phi \|_\Gamma^2, & 
        \text{subject to } \int_\Gamma \phi = \bar f.
    \end{align}

    To exemplify, if $\Gamma$ is the line segment $[0, 1]$, then the solution is given by the parabola $\phi(t) = 6 t (1 - t) \bar f$ with $0 \le t \le 1$. Continuous dependency on the data and the Poincaré inequality give us the bound
    \begin{align}
        \| \phi \|_{1, \Gamma} \lesssim | \bar f |.
    \end{align}

    We now define $\phi_h \in P_1(\Gamma_h)$ as an interpolation of $\phi$ onto the normal trace space of $\mathbf{V}_{S, h}$ on $\Gamma$. For the $H(\operatorname{div})$-conforming element of \Cref{sec-conforming}, we choose this in two steps. First, we use the Clément interpolant to determine the nodal degrees of freedom, which preserves homogeneous boundary conditions. We then use the degrees of freedom related to the $\mathbb{RT}_0$ basis functions to correct for mass conservation, i.e. they are chosen such that $\int_e \phi = \int_e \phi_h$ for all facets $e$ on $\Gamma_h$.  For the nonconforming element of \Cref{sec-nonconforming}, we set the degrees of freedom such that $\phi_h(x_e) = \frac1{|e|} \int_e \phi$ for each facet $e$ with midpoint $x_e$. We note that this again ensures that $\int_e \phi = \int_e \phi_h$.

    Next, we extend $\phi_h$ to the remaining boundaries. Let $\bm{\varphi}_{S,h} \in \mathbf{V}_{S,h}|_{\partial \Omega_S}$ and $\varphi_{D,h} \in \mathbf{n}_D \cdot \mathbf{V}_{D,h}|_{\partial \Omega_D}$ be given by: 
    \begin{align}
        \bm{\varphi}_{S,h} &:= \begin{cases}
            \phi_h \mathbf{n}_S, &\text{on }\Gamma, \\
            0, &\text{on }\partial \Omega_S \setminus \Gamma,
        \end{cases}
        &
        \varphi_{D,h} &:= \begin{cases}
            \Pi_{DS} \phi_h, &\text{on }\Gamma, \\
            0, &\text{on }\partial \Omega_D \setminus \Gamma,
        \end{cases}
    \end{align}
    where $\Pi_{DS}: \mathbf{n}_S \cdot \mathbf{V}_{S,h}|_{\Gamma} \to \mathbf{n}_S\cdot \mathbf{V}_{D,h}|_{\Gamma}$ is the interpolation operator of normal fluxes from the Stokes space to the Darcy space. For the nonconforming space, this is simply the identity operator.

    We are now ready to define the extensions $\mathbf{v}_S^\Gamma$ and $\mathbf{v}_D^\Gamma$ as the solutions to the following minimization problems, respectively:
    \begin{align*}
        \min_{\mathbf{v}_h \in \mathbf{V}_{S,h}}
        &\| \mathbf{v}_h \|_{V_{S,h}}^2
        + \| \nabla_h \cdot \mathbf{v}_h \|_S^2, &
        \text{subject to } \mathbf{v}_h|_{\partial \Omega_S} = \bm{\varphi}_{S,h}, \\
        \min_{\mathbf{v}_h \in \mathbf{V}_{D,h}}
        &\| \mathbf{v}_h \|_{V_{D,h}}^2
        + \| \nabla_h \cdot \mathbf{v}_h \|_D^2, &
        \text{subject to } \mathbf{n}_D \cdot \mathbf{v}_h|_{\partial \Omega_D} = -\varphi_{D,h}.
    \end{align*}
    \begin{subequations} \label{eq: bounds on vsvd}
    Since both minimization problems are well-posed, we obtain the following bounds
    \begin{align}
        \| \mathbf{v}_S^\Gamma \|_{V_{S,h}}
        + \| \nabla_h \cdot \mathbf{v}_S^\Gamma \|_S
        &\lesssim \| \bm{\varphi}_{S,h} \|_{1/2, \partial \Omega_S}
        \lesssim \| \phi \|_{1, \partial \Omega_S}
        = \| \phi \|_{1, \Gamma} 
        \lesssim |\bar f|, \\
        \| \mathbf{v}_D^\Gamma \|_{V_{D,h}}
        + \| \nabla_h \cdot \mathbf{v}_D^\Gamma \|_D
        &\lesssim \| \varphi_{D,h} \|_{-1/2, \partial \Omega_D}
        \lesssim \| \phi \|_{\Gamma}
        \lesssim |\bar f|. 
    \end{align}
    \end{subequations}

    Finally, we define $\mathbf{v}^\Gamma = (\mathbf{v}_S^\Gamma, \mathbf{v}_D^\Gamma)$. By construction, we have $\mathbf{v}^\Gamma \in \mathbf{V}_h$ since 
    \begin{align}
        \mathbf{n}_D \cdot \mathbf{v}_D^\Gamma
        &= - \varphi_{D,h}
        = - \Pi_{DS} \phi_h
        = - \Pi_{DS} \mathbf{n}_S \cdot \mathbf{v}_S^\Gamma, &
        \text{on } \Gamma.
    \end{align}
    The bound \eqref{eq: bound on vGamma} now follows from \eqref{eq: bounds on vsvd}. The property \eqref{eq: property vGamma} follows from the divergence theorem. Namely,
       $ (\nabla_h \cdot \mathbf{v}_S^\Gamma, 1)_S
        = (\mathbf{n}_S \cdot \mathbf{v}_S^\Gamma, 1)_{\partial \Omega_S}
        = (\phi_h , 1)_{\Gamma}
        = \bar f$.
    An analogous calculation holds for $\mathbf{v}_D^\Gamma$.
\end{proof}

We now use the construction from \Cref{lem: construction} to prove the remaining three cases. 

\begin{lemma}[\textbf{Case NE}]\label{lemma: NE}
   Under \cref{ass:ass1} and \cref{ass:ass2}, there exists a constant $\gamma > 0$ independent of $h$ such that,
    \begin{align*}
       &\inf_{q\in Q_h}\sup_{\mathbf{v}\in\mathbf{V}_h}\frac{b_h(\mathbf{v}, q) }{\|\mathbf{v}\|_{V_h}\|q\|_{Q_h}} \geq \gamma \min\left\{ \frac1{\sqrt{\mu {K}}}, 1 \right\}. 
    \end{align*}
\end{lemma}
\begin{proof}

Let $q=\left(q_S, q_D\right) \in Q_h$ be given. Since we have essential boundary conditions on the Darcy subdomain boundary, we first handle the mean of $q_D$. Let $\mathbf{v}^\Gamma = (\mathbf{v}_S^\Gamma, \mathbf{v}_D^\Gamma) $ be defined according to \Cref{lem: construction} with $\bar f = -(q_D, 1)_D$. 

Let us now consider the augmented pressure
\begin{align*}
    \tilde q_S &:= q_S - (\omega_S\mu)(\omega_D^{-1}{K}) P_{S, h} \nabla_h \cdot \mathbf{v}_S^\Gamma, &
    \tilde q_D &:= q_D - \nabla_h \cdot \mathbf{v}_D^\Gamma.
\end{align*}
We emphasize that $\tilde q_S \in Q_{S, h}$ and
$
    (\tilde q_D, 1)_D = 
    (q_D - \nabla_h \cdot \mathbf{v}_D^\Gamma, 1)_D = 
    (q_D, 1)_D + \bar f = 0,
$
which means that $\tilde q_D \in Q_{D, h} / \mathbb{R}$. Therefore, we can use \Cref{ass:ass1} and \Cref{ass:ass2} to construct suitable test functions $\mathbf{v}_S^0$ and $\mathbf{v}_D^0$ such that the analogues to \eqref{eqs: properties v0} hold, based on $\tilde q$ instead of $q$. Similar to \Cref{lemma: NN}, we set
\begin{align*}
    \mathbf{v} := \textcolor{red}{-}(\omega_S^{-1}\mu^{-1} \mathbf{v}_S^0, \omega_D^{-1}{K} \mathbf{v}_D^0) - \omega_D^{-1}{K} \mathbf{v}^\Gamma,
\end{align*}
where the minus signs are due to the definition of $b(\cdot, \cdot)$ in \eqref{bh-def}.
To derive bounds analogous to \eqref{eqs: infsup NN}, we begin by substituting the definitions into the bilinear form,
    \begin{align} \label{eq: substitution b NE}
        b_h(\mathbf{v}, q) &= 
        (\nabla_h \cdot (\omega_S^{-1}\mu^{-1} \mathbf{v}_S^0 + \omega_D^{-1}{K} \mathbf{v}_S^\Gamma), q_S)_S 
        + \omega_D^{-1}{K} (\nabla_h \cdot (\mathbf{v}_D^0 + \mathbf{v}_D^\Gamma), q_D)_D \nonumber \\
        &= 
        ( \omega_S^{-1}\mu^{-1} \tilde q_S + \omega_D^{-1}{K} P_{S, h} \nabla_h \cdot \mathbf{v}_S^\Gamma, q_S)_S 
        + \omega_D^{-1}{K} ( \tilde q_D + \nabla_h \cdot \mathbf{v}_D^\Gamma, q_D)_D \nonumber \\
        &= \omega_S^{-1}\mu^{-1} \| q_S \|_S^2 + \omega_D^{-1}{K} \| q_D \|_D^2 = \|q\|_{Q_h}^2.
    \end{align}
The above calculation provides the following estimates on the divergence terms
\begin{align} 
   & \omega_S\mu \| P_{S, h} \nabla_h \cdot \mathbf{v}_S \|_S^2 
    + \omega_D{K}^{-1} \| \nabla_h \cdot \mathbf{v}_D \|_D^2  \nonumber\\
 & = \omega_S^{-1}\mu^{-1} \| q_S \|_S^2 
    + \omega_D^{-1}{K} \| q_D \|_D^2 
    = \|q\|_{Q_h}^2. \label{eq: NE 1}
\end{align}

Next, we consider the Stokes component:
\begin{align} \label{eq: NE 2}
        \mu \| \mathbf{v}_S \|_{V_{S,h}}^2 
        &= 
        \mu \| \omega_S^{-1}\mu^{-1} \mathbf{v}_S^0 + \omega_D^{-1}{K} \mathbf{v}_S^\Gamma\|_{V_{S,h}}^2 
         \nonumber \\
        &
        \lesssim
        \omega_S^{-2}\mu^{-1} \| \mathbf{v}_S^0  \|_{V_{S,h}}^2 + \omega_D^{-2}\mu {K}^2 \| \mathbf{v}_S^\Gamma\|_{V_{S,h}}^2 
        \nonumber \\
        &\lesssim
        \omega_S^{-2}\mu^{-1} \Big( \| q_S \|_S^2 
        +\| (\omega_S\mu) (\omega_D^{-1}{K}) P_{S, h} \nabla_h \cdot \mathbf{v}_S^\Gamma  \|_S^2 \Big)
        + \omega_D^{-2}\mu {K}^2 \| \mathbf{v}_S^\Gamma\|_{V_{S,h}}^2 
        \nonumber \\
        &\lesssim
        \omega_S^{-2}\mu^{-1} \| q_S \|_S^2 
        + \omega_D^{-2}\mu {K}^2 \| q_D \|_D^2      
        \le 
        \max \{\mu {K}, 1 \} \|q\|_{Q_h}^2,
\end{align}
where we again used $\omega_S, \omega_D \ge 1$ in the final inequality.

Similarly, we bound the Darcy component as follows,
\begin{align} \label{eq: NE 3}
        {K}^{-1} \| \mathbf{v}_D \|_{V_{D,h}}^2 
        &= 
        \omega_D^{-2}{K} \| \mathbf{v}_D^0 + \mathbf{v}_D^\Gamma \|_{V_{D,h}}^2 
        \lesssim \omega_D^{-2}
        {K} \| \mathbf{v}_D^0 \|_{V_{D,h}}^2 + \omega_D^{-2}{K} \| \mathbf{v}_D^\Gamma \|_{V_{D,h}}^2 
        \nonumber \\
        &\lesssim
        \omega_D^{-2}{K} \| q_D \|_D^2 
        + \omega_D^{-2}{K} (\| \mathbf{v}_D^\Gamma\|_{V_{D,h}}^2 + \| \nabla_h \cdot \mathbf{v}_D^\Gamma  \|_D^2)
        \nonumber \\
        &\lesssim \omega_D^{-1}
        {K} \| q_D \|_D^2
        \le \|q_D\|_{Q_h}^2.
\end{align}
Collecting \eqref{eq: NE 1}, \eqref{eq: NE 2}, and \eqref{eq: NE 3}, we obtain
\begin{align*}
    \| \mathbf{v} \|_{V_h} \lesssim \max \left\{\sqrt{\mu {K}}, 1 \right\} \|q\|_{Q_h}.
\end{align*}
The combination of this bound with \eqref{eq: substitution b NE} yields the result.
\end{proof}

The proof for the \textbf{Case EN} follows along the same arguments, but a different scaling is obtained. 

\begin{lemma}[\textbf{Case EN}]\label{lemma: EN}
   Under \cref{ass:ass1} and \cref{ass:ass2}, there exists a constant $\gamma > 0$ independent of $h$ and material parameters such that,
    \begin{align*}
       &\inf_{q\in Q_h}\sup_{\mathbf{v}_h\in\mathbf{V}_h}\frac{b_h(\mathbf{v}, q) }{\|\mathbf{v}\|_{V_h}\|q\|_{Q_h}} \geq 
       \gamma \min \left\{\sqrt{\mu {K}}, 1\right\}. 
    \end{align*}
\end{lemma}
\begin{proof}
We outline the main idea, as the proof follows the same process as in \Cref{lemma: NE}, with $\mathbf{v}_S$ and $\mathbf{v}_D$ switching roles. Specifically, for a given $(q_S, q_D) \in Q_{h}$, we first construct $\mathbf{v}^\Gamma$ using \Cref{lem: construction} with $\bar{f} = (q_S, 1)_S$. Then the augmented pressures are given by
\begin{align*}
    \tilde q_S &:= q_S - P_{S, h} \nabla_h \cdot \mathbf{v}_S^\Gamma, &
    \tilde q_D &:= q_D - (\omega_S^{-1}\mu^{-1}) (\omega_D {K}^{-1}) \nabla_h \cdot \mathbf{v}_D^\Gamma. 
\end{align*}
These are used to construct $\mathbf{v}_S^0$ and $\mathbf{v}_D^0$ according to \cref{ass:ass1} and \cref{ass:ass2}.  Then we define
$
    \mathbf{v} := -(\omega_S^{-1}\mu^{-1} \mathbf{v}_S^0, \omega_D^{-1}{K} \mathbf{v}_D^0) - \omega_S^{-1}\mu^{-1} \mathbf{v}^\Gamma.
$
Following the same steps as in \Cref{lemma: NE}, we now obtain
\begin{align*} 
    b_h(\mathbf{v}, q) &= \|q\|_{Q_h}^2, &
    \| \mathbf{v}_h \|_{V_h} &\lesssim \max \left\{ \frac1{\sqrt{\mu {K}}}, 1 \right\} \|q\|_{Q_h},
\end{align*}
from which the result follows.
\end{proof}

Lastly, we present \textbf{Case EE} in the following lemma.
\begin{lemma}[\textbf{Case EE}] \label{lemma: EE}
   Under \cref{ass:ass1} and \cref{ass:ass2}, there exists a constant $\gamma > 0$ independent of $h$ and material parameters such that,
    \begin{align}
       &\inf_{q\in Q_h / \mathbb{R}}\sup_{\mathbf{v}_h\in\mathbf{V}_h}\frac{b_h(\mathbf{v}, q) }{\|\mathbf{v}\|_{V_h}\|q\|_{Q_h}} \geq 
       \gamma.
    \end{align}
\end{lemma}
\begin{proof}
    In this case, both subproblems have essential boundary conditions and we assume a given pressure $q = (q_S, q_D) \in Q_h / \mathbb{R}$. In other words, we have
    \begin{align}
        (q_S, 1)_S = - (q_D, 1)_D.
    \end{align}
    The construction from \Cref{lem: construction} therefore produces a $\mathbf{v}^\Gamma$ that handles both constants simultaneously. The subsequent steps proceed exactly as in \Cref{lemma: NE} or \Cref{lemma: EN}, with the only difference being the scaling in the augmented pressures and test functions. If $\mu {K} \le 1$, we follow the strategy from \Cref{lemma: NE}; otherwise, if $\mu {K} > 1$, we use the construction from \Cref{lemma: NE}. In both cases, we obtain a lower bound that is independent of material parameters.
\end{proof}

Under the above assumptions and lemmas, the following theorem shows the well-posedness of the coupled linear system \eqref{finite-element-scheme}.
\begin{theorem}\label{thm:well-posedness}
     The saddle point problem \eqref{finite-element-scheme} is well-posed in $\mathbf{V}_h\times Q_h$, endowed with the norm \eqref{norm}.
     The stability and continuity constants are independent of $\mu$ and ${K}$ for \textbf{Case NN} and \textbf{Case EE}. However, for \textbf{Case NE} and \textbf{Case EN}, the stability constant depends on the product $\mu {K}$, as illustrated in \Cref{lemma: NE} and \Cref{lemma: EN}.
\end{theorem}

The following corollaries show the well-posedness results for the two finite-element examples from \Cref{sec-conforming,sec-nonconforming}. According to the proof of \cref{lemma: NE}, we just need to verify \cref{ass:ass1} and \cref{ass:ass2}.  
\begin{corollary}\label{corollary-1}
For the $H(\operatorname{div})$-conforming element in \cref{sec-conforming}, the system \eqref{finite-element-scheme} is well-posed in $\mathbf{V}_h\times Q_h$.
\end{corollary}
\begin{proof}
    For \cref{ass:ass1},
    $\mathbb{P}_1\oplus\mathbb{RT}_0$-$\mathbb{P}_0$ is a Stokes-stable pair (cf.~\cite[Sec.~3.3]{Li2022A}). Specifically, for $\mathbf{v}_S|_{\Gamma} = 0$, according to \cite[Lem.~3.6]{Li2022A}, there holds
    \begin{align*}
        & \|\mathbf{v}_S\|_{V_{S,h}}^2 = 2(\bm{\epsilon}(\mathbf{u}_S^l), \bm{\epsilon}(\mathbf{u}_S^l ))_S  +2 \sum_{T \in \mathcal{T}_h^S} \sum_{e \in \mathcal{E}_h(\bar{\Omega}_S)} \alpha_T u_e v_e(\nabla \cdot \bm{\phi}_e, \nabla \cdot \bm{\phi}_e)_T \leq C \|\mathbf{v}_S\|_{1}^2,
    \end{align*}
    where $C$ is independent of the physical and material parameters. So the inf-sup condition \eqref{inf-sup-Stokes} holds. 
    
    On the other hand, \cref{ass:ass2} is satisfied because the Raviart-Thomas is a Darcy-stable pair \cite[Sec.~7.1.2]{boffi2013mixed}. 
    Therefore, the $H(\operatorname{div})$-conforming finite-element method is stable with respect to the weighted norm \cref{norm}.
\end{proof}
\begin{corollary}\label{corollary-2}
For the nonconforming element in \cref{sec-nonconforming}, the system \eqref{finite-element-scheme} is well-posed in $\mathbf{V}_h\times Q_h$.
\end{corollary}
\begin{proof}
     For \cref{ass:ass1}, $\mathbb{CR}_0$-$\mathbb{P}_0$ is a stable Stokes pair (cf.~\cite{Crouzeix1973Conforming}), and for $\mathbf{v}_S|_{\Gamma} = 0$, it is clear that 
     \begin{align*}
        & \|\mathbf{v}_S\|_{V_{S,h}}^2   = 2\sum_{T\in\mathcal{T}_h^S}(\bm{\epsilon}(\mathbf{v}_S),\bm{\epsilon}(\mathbf{v}_S))_T + \sum_{e\in\mathcal{E}_h^0(\Omega_S)} \int_{e} \frac{1}{h_e} [\![\mathbf{v}_S]\!]_e ~ [\![\mathbf{v}_S]\!]_e \mathrm{~d}s \leq C\|\mathbf{v}_S\|_{1}^2,
     \end{align*}
     where $C$ is independent of the physical and discretization parameters. The inf-sup condition \eqref{inf-sup-Stokes} holds.

     For \cref{ass:ass2}, note that all Stokes-stable pairs are Darcy stable. This is because the (broken) $H(\operatorname{div})$ norm is weaker than the discrete $H^1$-type norm $\| \cdot \|_{V_{S,h}}$. Thus, \cref{ass:ass2} follows from \Cref{ass:ass1} in this case.
   In conclusion, the nonconforming finite-element method is well-posed with respect to the norm \cref{norm}.
\end{proof}

\section{Parameter-Robust Preconditioner}\label{sec-preconditioner}
In this section, we develop a robust preconditioner for the linear system based on its well-posedness, leveraging the operator preconditioning framework~\cite{doi:10.1137/S1064827502418203,Kent2011Preconditioning}. The key idea of this framework is that the well-posedness of the system ensures that the linear operator is an isomorphism between the original Hilbert space and its dual. Consequently, any isomorphism from the dual space back to the Hilbert space can serve as an effective preconditioner.

A natural choice for such an isomorphism on a Hilbert space $\mathbf{X}$ is the Riesz operator $\mathcal{B}: \mathbf{X}' \to \mathbf{X}$ defined by $\langle \mathcal{B}^{-1} \mathbf{x}, \mathbf{y} \rangle = ( \mathbf{x}, \mathbf{y})_{\mathbf{X}}$, $\forall \,\mathbf{x}, \mathbf{y} \in \mathbf{X}$,
where $\langle \cdot, \cdot \rangle$ denotes the duality pairing between $\mathbf{X}'$ and $\mathbf{X}$ and $(\cdot, \cdot)_{\mathbf{X}}$ is the inner product on $\mathbf{X}$.

In our case, we have $\mathbf{X} = \mathbf{V}_h\times Q_h$ and let $(\cdot, \cdot)_{\mathbf{X}}$ be the inner product that induces the weighted norm $\enorm{\cdot}$ from \eqref{norm}. Using the operator preconditioning framework, we now propose the following $2\times 2$ block diagonal preconditioner,
\begin{equation}\label{preconditioner}
	\mathcal{B} = 
	\left(
	\begin{array}{cc}
		A_{\mathbf{u}} + D_{\mathbf{u}} &  \\
		& I_p  \\
	\end{array}
	\right)^{-1},
\end{equation}
with $A_{\mathbf{u}}$ from \eqref{eq: saddlepoint A} and the remaining blocks corresponding to
\begin{align*} 
  \omega_S\mu\|P_{S,h}(\nabla_h\cdot \mathbf{u}_S)\|_S^2 + \omega_D{K}^{-1}\|\nabla_h\cdot \mathbf{u}_D\|^2_{D}  &\mapsto D_{\mathbf{u}},\\ 
	\omega_S^{-1}\mu^{-1}\|p_S\|_S^2 + \omega_D^{-1}{K}\|p_D\|_{D}^2 &\mapsto I_p.
\end{align*}
We emphasize that the preconditioner $\mathcal{B}$ does not involve any fractional parts. 

For the analysis, we introduce the short-hand notation $\mathbf{x} = (\mathbf{u},p), \mathbf{y} = (\mathbf{v},q)$. Then the operator $\mathcal{A}: \mathbf{X} \to \mathbf{X}'$ from \cref{eq: saddlepoint A} corresponds to the composite bilinear form
\begin{align*}
  \langle \mathcal{A} \mathbf{x}, \mathbf{y} \rangle 
  = a_h(\mathbf{u},\mathbf{v}) + b_h(\mathbf{v},p) + b_h(\mathbf{u},q).
\end{align*}
\cref{thm:well-posedness} then shows that the operator $\mathcal{A}$ satisfies the following bounds
\begin{align}
  \sup_{\mathbf{x, y}\in \mathbf{X}} \langle \mathcal{A} \mathbf{x}, \mathbf{y} \rangle &\leq C \enorm{\mathbf{x}} \enorm{\mathbf{y}}, &
  \inf_{\mathbf{x}\in \mathbf{X}} \sup_{\mathbf{y}\in \mathbf{X}} \frac{\langle \mathcal{A} \mathbf{x}, \mathbf{y} \rangle}{\enorm{\mathbf{x}} \enorm{\mathbf{y}} } &\geq \gamma, \label{bounds A}
\end{align}
for constants $0 < \gamma \le C < \infty$.

We begin by analyzing \textbf{Case NN} and \textbf{Case EE}, as \cref{thm:well-posedness} establishes that the constants in \eqref{bounds A} are independent of both physical and discretization parameters in these cases. Consequently, following the operator preconditioning framework, the Riesz operator \eqref{preconditioner} is parameter-robust. This result is summarized as follows. 

\begin{theorem}\label{thm:ConditionNumber-Case1&2}
For \textbf{Case NN} and \textbf{Case EE}, under \cref{ass:ass1} and \cref{ass:ass2},  the condition number of the preconditioned operator satisfies 
\begin{equation*}
\kappa(\mathcal{BA}) :=  \|\mathcal{B}\mathcal{A}\|_{\mathbf{X}\mapsto\mathbf{X}}  \|(\mathcal{B}\mathcal{A})^{-1}\|_{\mathbf{X}\mapsto\mathbf{X}}   = \mathcal{O}(1).
 \end{equation*}
\end{theorem}

\begin{proof}
Based on \eqref{bounds A}, we have
\begin{align*}
	&\|\mathcal{B}\mathcal{A}\|_{\mathbf{X}\mapsto\mathbf{X}}
  := \sup_{\mathbf{x} \in\mathbf{X}} \frac{ \enorm{\mathcal{B}\mathcal{A}\bm{x}}}{ \enorm{\bm{x}}} 
  = \sup_{\mathbf{x},\mathbf{y}\in\mathbf{X}} \frac{(\mathcal{B}\mathcal{A}\bm{x},\bm{y})_{\mathbf{X}}}{ \enorm{\bm{x}} \enorm{\bm{y}}} 
	= \sup_{\mathbf{x},\mathbf{y}\in\mathbf{X}}\frac{\langle \mathcal{A} \mathbf{x}, \mathbf{y} \rangle}{ \enorm{\bm{x}} \enorm{\bm{y}}}
	\leq C,
\end{align*}
and 
\begin{align*}
	\|(\mathcal{B}\mathcal{A})^{-1}\|^{-1}_{\mathbf{X}\mapsto\mathbf{X} } & = \inf_{\bm{x}\in\mathbf{X}}\frac{\enorm{\mathcal{B}\mathcal{A}\bm{x}} }{ \enorm{\bm{x}} } 
	= \inf_{\bm{x}\in\mathbf{X}}\sup_{\bm{y}\in\mathbf{X}} \frac{(\mathcal{B}\mathcal{A}\bm{x},\bm{y})_{\mathbf{X}}}{ \enorm{\bm{x}}  \enorm{\bm{y}} } = \inf_{\bm{x}\in\mathbf{X}}\sup_{\bm{y}\in\mathbf{X}} \frac{\langle \mathcal{A} \mathbf{x}, \mathbf{y} \rangle}{ \enorm{\bm{x}}  \enorm{\bm{y}} } \geq \gamma. \nonumber
\end{align*}
Thus, 
$\kappa(\mathcal{BA}) = \|\mathcal{B}\mathcal{A}\|_{\mathbf{X}\mapsto\mathbf{X}}  \|(\mathcal{B}\mathcal{A})^{-1}\|_{\mathbf{X}\mapsto\mathbf{X}} \leq C/\gamma = \mathcal{O}(1)$.
\end{proof} 

For the \textbf{Case NE} and \textbf{Case EN}, as shown in \cref{lemma: NE,lemma: EN}, the inf-sup constant depends on $\mu$ and $K$, though it remains independent of the mesh size. Consequently, $\kappa(\mathcal{BA})$ depends on $\mu$ and $K$. However, a closer examination of the proofs of \cref{lemma: NE,lemma: EN} reveals that, by considering the quotient spaces 
\begin{align*}
  \mathbf{X} / \mathbb{R}_S &:= \mathbf{V}_h \times (Q_{S, h} / \mathbb{R} , Q_{D, h}), & 
  \mathbf{X} / \mathbb{R}_D &:= \mathbf{V}_h \times (Q_{S, h} , Q_{D, h} / \mathbb{R}),
\end{align*}
the inf-sup constant becomes independent of the physical parameters again, ensuring parameter-robustness. This result is formalized in the following lemma.

\begin{lemma}\label{lem:special case}
	The inf-sup constant $\gamma$ in \eqref{bounds A} is independent of $\mu$ and ${K}$ on $\mathbf{X} / \mathbb{R}_D$ for \textbf{Case NE} and, similarly, on $\mathbf{X} / \mathbb{R}_S$ for \textbf{Case EN}.
\end{lemma}
\begin{proof}
	We start with \textbf{Case NE} and consider the subspace $ \mathbf{X} / \mathbb{R}_D$. In this case, we have $(q_D, 1)_D = 0$ and, thus, in the proof of \Cref{lemma: NE}, we have $\mathbf{v}^\Gamma = 0$. As a result, the proof proceeds identically to that for \cref{lemma: NN}, yielding a parameter-independent inf-sup constant for the bilinear form $b_h$. 

  The proof of \textbf{Case EN} follows analogously on the subspace $\mathbf{X} / \mathbb{R}_S$. As a result, the details are therefore omitted. 
\end{proof}

\cref{lem:special case} allows us to analyze the effective condition number in the quotient space, as described in the following theorem. 

\begin{theorem}\label{thm:ConditionNumber-Case3&4}
	For \textbf{Case NE}, under \cref{ass:ass1} and \cref{ass:ass2}, the effective condition number of the preconditioned operator satisfies 
	\begin{equation*}
		\kappa_{\mathrm{eff}}(\mathcal{BA}) :=  \|\mathcal{B}\mathcal{A}\|_{\mathbf{X} / \mathbb{R}_D\mapsto\mathbf{X} / \mathbb{R}_D}  \|(\mathcal{B}\mathcal{A})^{-1}\|_{\mathbf{X} / \mathbb{R}_D\mapsto\mathbf{X} / \mathbb{R}_D}   = \mathcal{O}(1),
	\end{equation*}
	where the quotient space $\mathbf{X} / \mathbb{R}_D$ is endowed with the same norm $\enorm{\cdot}$ from \eqref{norm}.
  The analogous result holds for \textbf{Case EN} on $\mathbf{X} / \mathbb{R}_S$.
\end{theorem}
\begin{proof}
Based on \cref{lem:special case}, the inf-sup condition of $\mathcal{A}$ becomes
\begin{align*}
\inf_{(\mathbf{u},p)\in\mathbf{X} / \mathbb{R}_D} \sup_{{(\mathbf{v},q)\in\mathbf{X} / \mathbb{R}_D}} \frac{\mathcal{A}(\mathbf{u},p;\mathbf{v},q)}{\enorm{(\mathbf{u},p)} \enorm{(\mathbf{v},q)} } \geq \gamma, 
\end{align*}
where $\gamma$ now is parameter-robust.  Consequently, following the same proof of \cref{thm:ConditionNumber-Case1&2}, we can show that $\kappa_{\mathrm{eff}}(\mathcal{BA})  \leq  C/\gamma$.
\end{proof}

By the Courant–Fischer–Weyl min-max principle, \cref{lem:special case} and \cref{thm:ConditionNumber-Case3&4} indicate that we can divide $\sigma(\mathcal{BA})$, the spectrum of $\mathcal{BA}$, into two parts: $\{\lambda_0\}$ and $\sigma_1(\mathcal{BA})$.  The eigenvalue $\lambda_0$, which has the smallest absolute value, depends on the physical parameters $\mu$ and ${K}$. On the other hand, $\sigma_1(\mathcal{BA}) \subset [-C, -\gamma] \cup [\gamma, C]$ contains the remaining eigenvalues that are parameter-robust.  Based on the theory of the Minimal Residual method \cite{simoncini2013superlinear}, the convergence rate is dominated by $(\kappa_{\mathrm{eff}}(\mathcal{BA})-1)/(\kappa_{\mathrm{eff}}(\mathcal{BA})+1)$, which is parameter-robust.  Thus, $\mathcal{B}$ can still be considered as a parameter-robust preconditioner for \textbf{Case NE} and \textbf{Case EN}, as confirmed by our numerical experiments in \Cref{sec:numerical}.

\section{Numerical Examples}\label{sec:numerical}

In this section, we investigate convergence of the finite-element methods and evaluate the parameter-robustness of the proposed preconditioners. All computations are performed on a laptop equipped with an 11th Gen Intel(R) Core(TM) i5-1135G7 processor and 16GB RAM, using MATLAB.R2024b. 

For each test, we use GMRes to solve the linear system arising from both the $H(\operatorname{div})$-comforming $\mathbb{P}_1\oplus\mathbb{RT}_0-\mathbb{P}_0-\mathbb{RT}_0-\mathbb{P}_0$ discretization, and the nonconforming $\mathbb{CR}_0-\mathbb{P}_0-\mathbb{CR}_0-\mathbb{P}_0$ discretization. We set the maximum number of iterations to $1000$ and the stopping criterion to $\| \mathcal{A}\mathbf{x}_n - \mathbf{b}\|_{\ell^2}\leq 10^{-6}\|\mathbf{b}\|_{\ell^2}$. The initial guess is always set to $\mathbf{0}$.   We choose the parameters $\omega_S = 100$ and $\omega_D = 1$ (cf.~\cref{eqs: norms v and Q}) in all the experiments, and set the BJS parameter in \cref{interface-3} as $\alpha_{BJS} = 1$. 
In this subsection, we test the errors and convergence rates.  The computational domain consists of the Stokes domain $\Omega_S = [0, 1] \times [0, 1]$, the Darcy domain $\Omega_D = [1, 2]\times[0, 1]$, and the interface $\Gamma = \{1\} \times (0, 1)$. The exact solutions are given by
\begin{align*}
\mathbf{u}_S &= 
\left(
  \begin{array}{c}
    (e^{x}-e)\cos(\pi y) \\
    -e^{x}\sin(\pi y)/\pi \\
  \end{array}
\right),
& p_S &= 2e^{x}\cos(\pi y), & \text{ in } &\Omega_S, \\
\mathbf{u}_D &= 
\left(
  \begin{array}{c}
    -(e^{x}-e)\cos(\pi y) \\
    \pi(e^x -xe)\sin(\pi y) \\
  \end{array}
\right),
& p_D &= (e^x - xe)\cos(\pi y), & \text{ in } &\Omega_D.
\end{align*}
The right-hand sides $\mathbf{f}_S, \mathbf{f}_D$ and the boundary conditions are determined by substituting the chosen physical parameters and the exact solution into the Stokes-Darcy problem. For simplicity, we present the results only for Stokes-Darcy coupling problem with essential boundary condition, i.e. \textbf{Case EE}. 
\subsection{Convergence Rates}
\begin{figure}[ht]
\centering
\includegraphics[width=\textwidth]{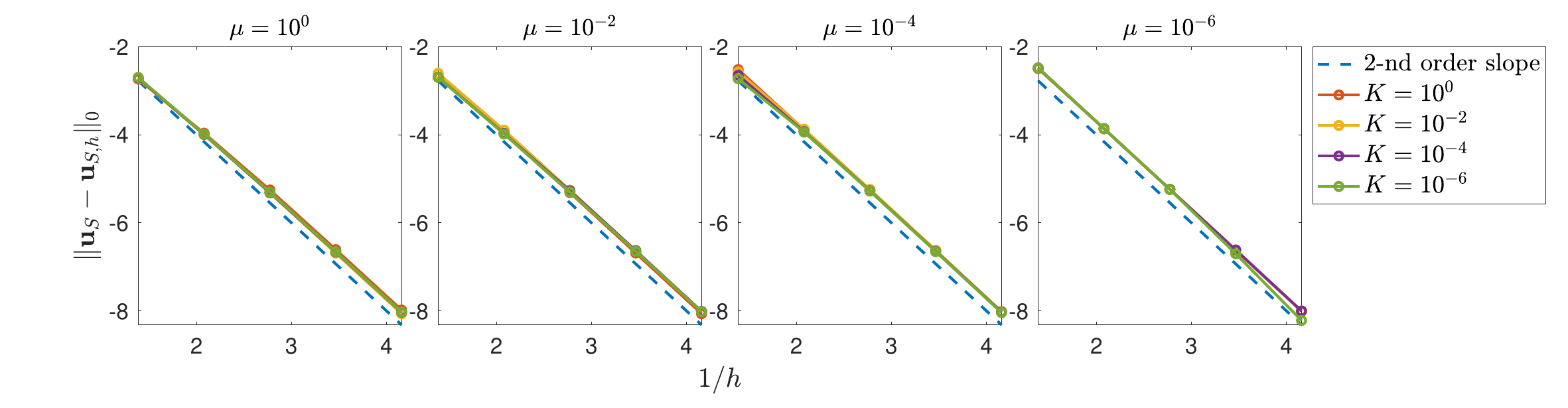}\\
\includegraphics[width=\textwidth]{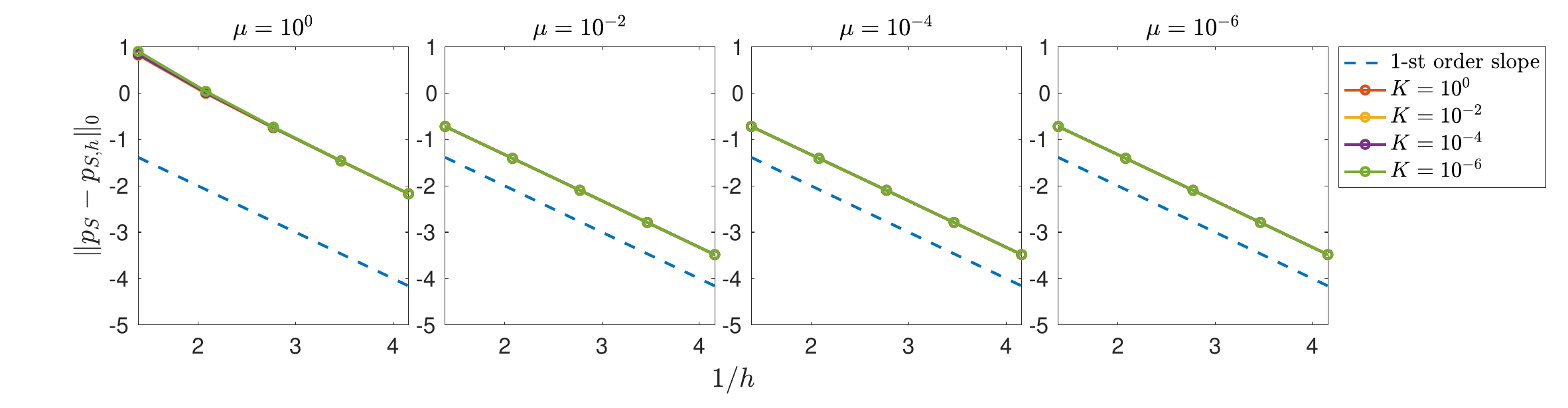}\\
\includegraphics[width=\textwidth]{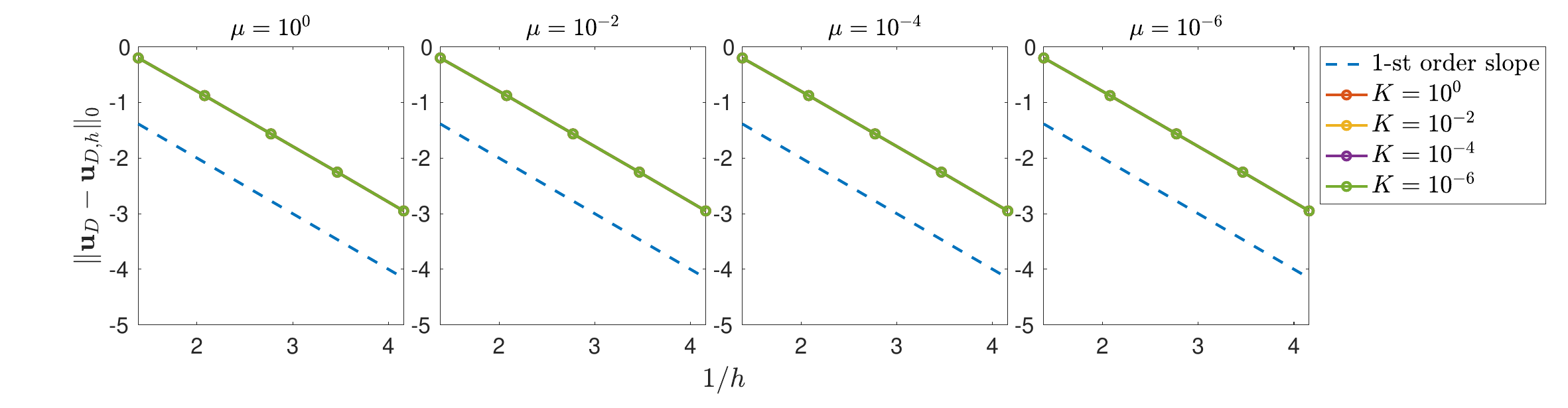}\\
\includegraphics[width=\textwidth]{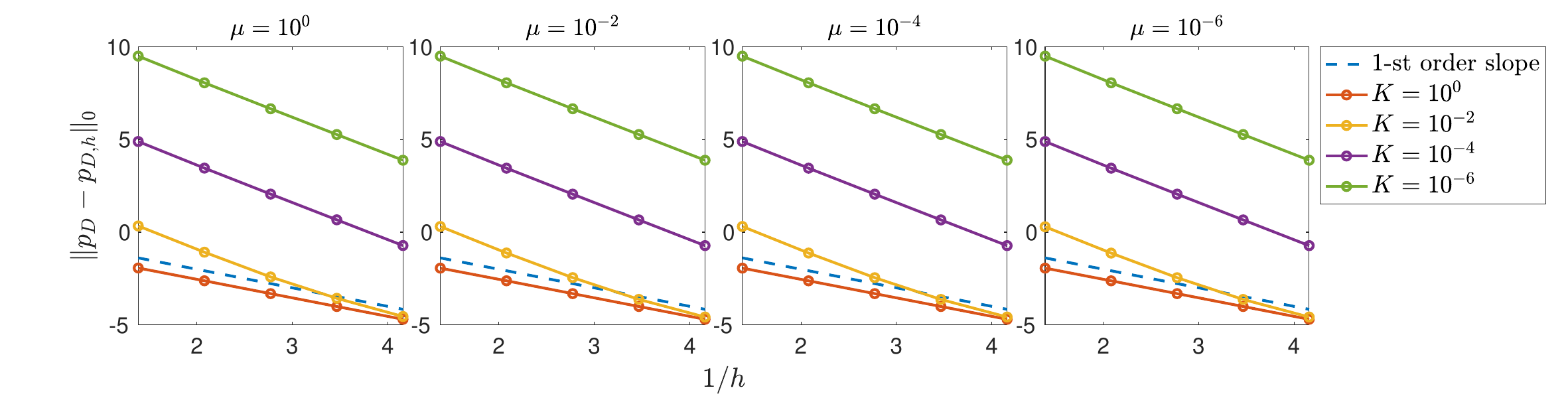}\\
\caption{Errors and convergence rate using $H(\operatorname{div})$-conforming elements (displayed in log-log scale). The horizontal axis corresponds the reciprocal of the mesh size $h$, while the vertical axis represents the errors measured in different norms.}
\label{fig:conforming}
\vskip -20pt
\end{figure}
\begin{figure}
\centering
\includegraphics[width=\textwidth]{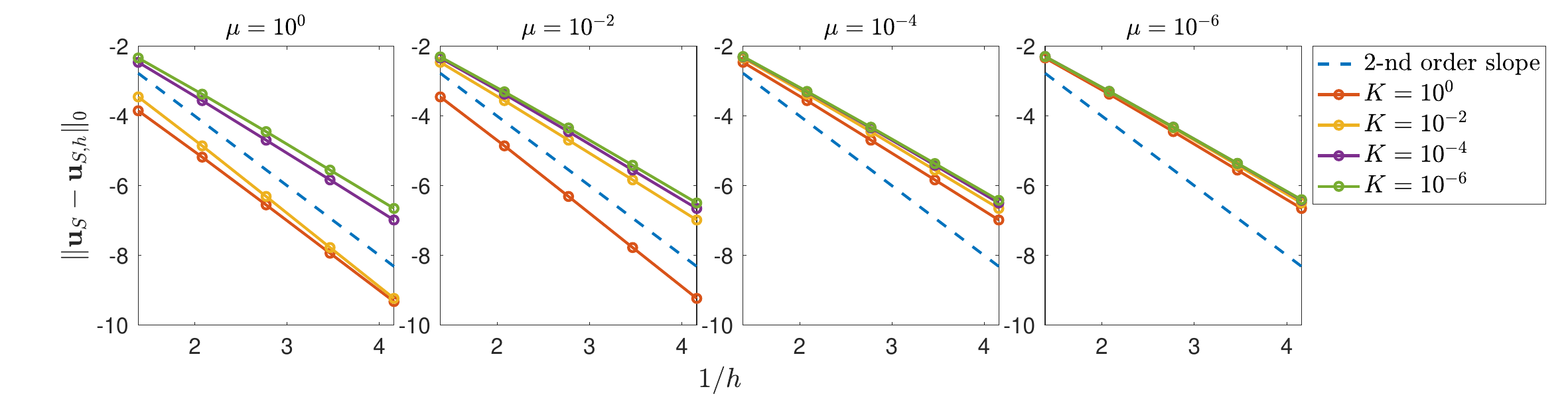}\\
\includegraphics[width=\textwidth]{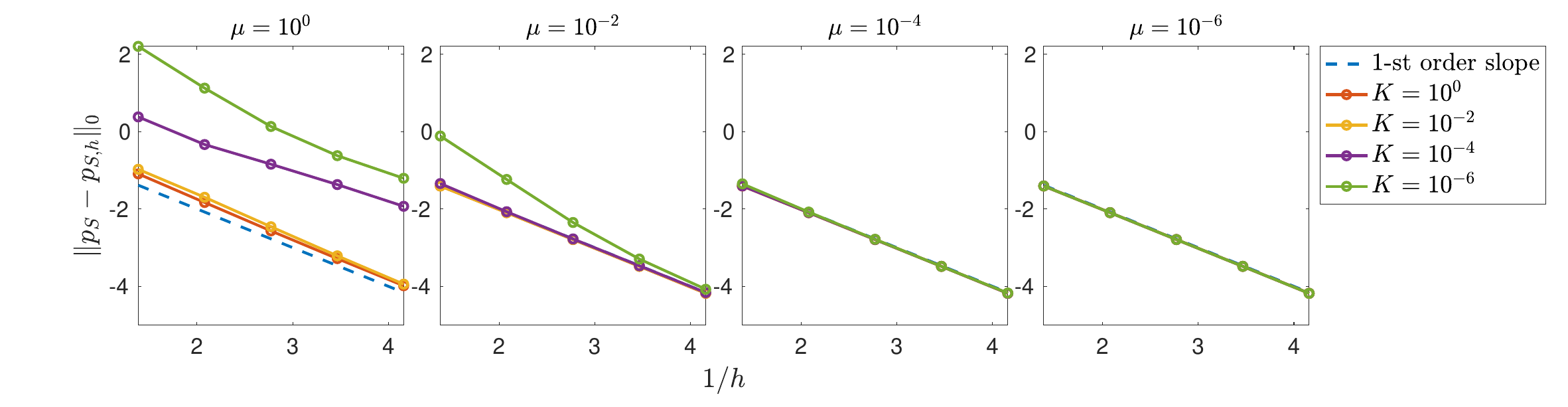}\\
\includegraphics[width=\textwidth]{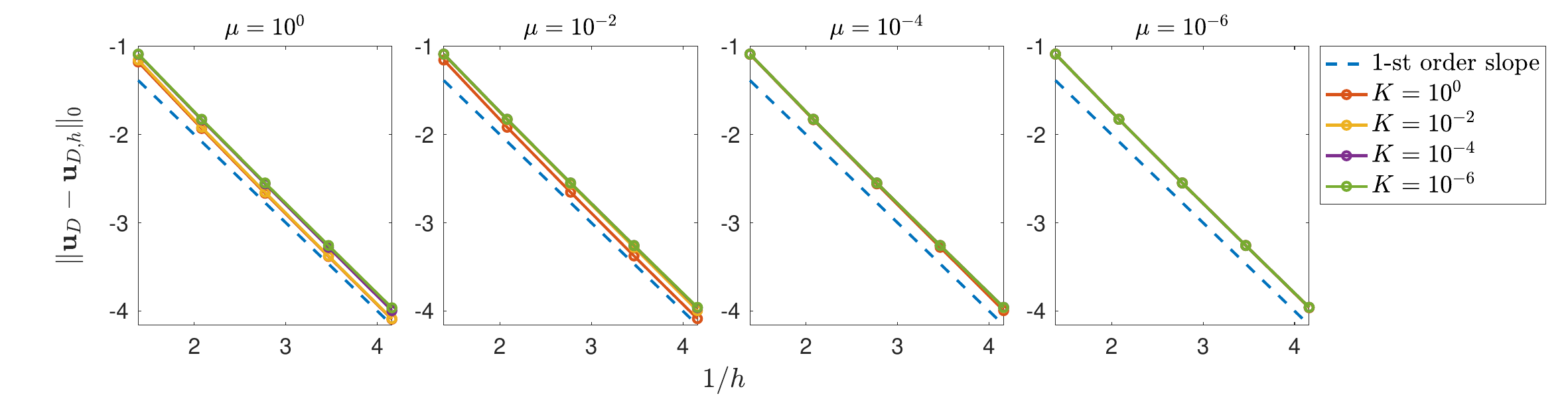}\\
\includegraphics[width=\textwidth]{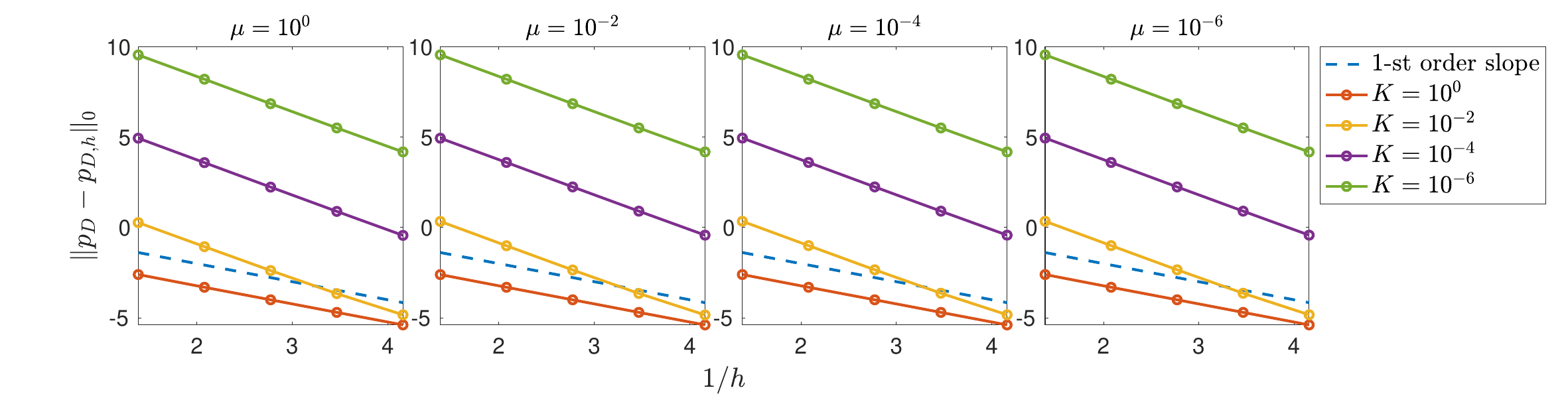}\\
\caption{Errors and convergence rate using nonconforming elements (displayed in log-log scale). The horizontal axis corresponds the reciprocal of the mesh size $h$, while the vertical axis represents the errors measured in different norms.
}
\label{fig:nonconforming}
\vskip -20pt
\end{figure}

The errors and convergence rates for the $H(\operatorname{div})$-conforming elements and non-conforming elements are shown in \cref{fig:conforming,fig:nonconforming}, respectively. We measure the errors and convergence rates in various norms using mesh size $h = 1/4, 1/8, 1/16, 1/32$, and $ 1/64$, while varying the viscosity $\mu$ and hydraulic conductivity $K$ in the ranges $10^{-6} \leq \mu \leq 1$ and $10^{-6} \leq K \leq 1$. For the $H(\operatorname{div})$-conforming elements, \cref{fig:conforming} shows that the pressures converge linearly in the $L^2$ norm. Moreover, the velocity exhibits quadratic convergence in $\Omega_S$, whereas in the Darcy domain it is first order, indicating that the convergence rates are optimal. For the nonconforming elements, \cref{fig:nonconforming} illustrates that the errors of the Darcy velocity and both pressures converge linearly. For the Stokes velocity, when $K$ becomes smaller, the rate deteriorates to approximately $\frac32$. The optimal rates can be recovered by using a larger penalty parameter for the jump term. However, such a study is beyond the scope of this paper, and we simply take the parameter to be $1$.

\subsection{Parameter-robustness} 
In this section, we study the performance of the preconditioner with respect to the mesh size $h$ and physical parameters $\mu$ and $K$, to validate the analysis in \cref{sec-preconditioner}. 
The diagonal preconditioner $\mathcal{B}$ \eqref{preconditioner} is implemented with a direct solver, i.e., its diagonal blocks, $(A_{\mathbf{u}}+D_{\mathbf{u}})^{-1}$ and $I_p^{-1}$, are computed using MATLAB's build-in ``$\backslash$'' function. 

We first investigate the dependence of the preconditioner on the mesh size. Setting $\mu = K = 1$, we begin with a coarse mesh of $h = 1/8$. The mesh is uniformly refined four times, and for each refinement, we present the condition numbers and the number of iterations (in parentheses) for four different boundary conditions, see \cref{tab:cond-h}. From the relatively consistent iteration counts and condition numbers, we observe that the preconditioned system is robust with respect to the discretization parameters.

\begin{table}[h]
\renewcommand{\arraystretch}{1.1}
\centering
\caption{The condition numbers and iteration counts (in parentheses) for different mesh sizes $h$, with material parameters fixed at $\mu = K = 1$.}   \label{tab:cond-h}
\begin{tabular}{|c|ccccc|} 
\hline  \hline 
\multicolumn{6}{|c|}{$H(\operatorname{div})$-conforming element}\\
\hline
$\text{Case} \backslash h$ & $2^{-3}$ & $2^{-4}$ & $2^{-5}$ & $2^{-6}$ & $2^{-7}$ \\
\hline 
\textbf{NN}  & 1.08(5) & 1.08(5) & 1.08(5) & 1.08(5) & 1.08(5) \\
\textbf{EE}  & 1.14(6) & 1.14(6) & 1.14(6) & 1.14(6) & 1.14(6) \\
\textbf{NE}  & 6.11(7) & 6.02(7) & 5.99(7) & 5.98(7) & 5.97(6) \\
\textbf{EN}  & 1.14(6) & 1.14(6) & 1.14(6) & 1.14(6) & 1.14(6) \\
\hline \hline
\multicolumn{6}{|c|}{Nonconforming element}\\
\hline
$\text{Case} \backslash h$ & $2^{-3}$ & $2^{-4}$ & $2^{-5}$ & $2^{-6}$ & $2^{-7}$ \\
\hline 
\textbf{NN} & 1.08(5) & 1.08(5) & 1.08(5) & 1.08(5) & 1.08(4) \\
\textbf{EE} & 1.10(5) & 1.11(5) & 1.11(5) & 1.11(5) & 1.11(5) \\
\textbf{NE} & 5.82(7) & 5.90(7) & 5.94(7) & 5.96(7) & 5.97(6) \\
\textbf{EN} & 1.09(5) & 1.09(5) & 1.09(5) & 1.09(5) & 1.09(5)\\
\hline \hline
\end{tabular}
\end{table}

We continue by investigating the robustness of preconditioner with respect to physical parameters. In this case, we vary both the (scalar) hydraulic conductivity $K$ and the viscosity $\mu$ over a range of eight orders of magnitude, for a fixed mesh size $h = 1/32$. \cref{tab-conf:cond-case12}-\cref{table:cond-case4} show the iteration counts which are shown in parentheses for all cases for the two discretizations. It is apparent that the preconditioner is robust with respect to both parameters, reaching the desired tolerance within a maximum of eight iterations for the two test cases. 
\begin{table}
\renewcommand{\arraystretch}{1.1}
\centering
\caption{The condition numbers and iteration counts (in parentheses) for \textbf{Case NN} and \textbf{Case EE} using $H(\operatorname{div})$-conforming element with a mesh size of $h = 1/32$.}  
\label{tab-conf:cond-case12}
\begin{tabular}{|l|c|ccccc|}
\hline \hline 
\multicolumn{2}{|c|}{\multirow{2}{*}{\textbf{NN}}}  & \multicolumn{5}{c|}{$\log_{10}(\mu)$} \\
\cline{3-7} \multicolumn{2}{|c|}{} & -4 & -2 & 0 & 2 & 4 \\
\hline 
\multirow{5}{*}{\rotatebox[origin=c]{90}{$\log_{10}(K)$}}
& 4      & 1.08(5) & 1.08(5) & 1.08(5) & 1.08(5) & 1.08(5) \\
& 2      & 1.06(5) & 1.08(5) & 1.08(5) & 1.08(5) & 1.08(5) \\
& 0      & 1.05(5) & 1.06(5) & 1.08(5) & 1.08(5) & 1.08(5) \\
& -2     & 1.05(5) & 1.05(5) & 1.06(5) & 1.08(5) & 1.08(5) \\
& -4     & 1.05(5) & 1.05(5) & 1.05(5) & 1.06(5) & 1.08(5) \\
\hline
\hline 
\multicolumn{2}{|c|}{\multirow{2}{*}{\textbf{EE}}}  & \multicolumn{5}{c|}{$\log_{10}(\mu)$} \\
\cline{3-7} \multicolumn{2}{|c|}{} & -4 & -2 & 0 & 2 & 4 \\
\hline 
\multirow{5}{*}{\rotatebox[origin=c]{90}{$\log_{10}(K)$}}
& 4      & 1.14(6) & 1.14(6) & 1.14(6) & 1.14(6) & 1.14(6) \\
& 2      & 1.27(6) & 1.14(6) & 1.14(6) & 1.14(6) & 1.14(6) \\
& 0      & 1.40(6) & 1.27(6) & 1.14(6) & 1.14(6) & 1.14(6) \\
& -2     & 1.41(6) & 1.40(6) & 1.27(6) & 1.14(6) & 1.14(6) \\
& -4     & 1.41(6) & 1.41(6) & 1.40(6) & 1.27(8) & 1.14(6) \\
\hline \hline 
\end{tabular}
\end{table}

\begin{table}
\renewcommand{\arraystretch}{1.1}
\centering
\caption{The condition number $\kappa$, the effective condition number $\kappa_{\mathrm{eff}}$, and the iteration counts (in parentheses) for \textbf{Case NE} using $H(\operatorname{div})$-conforming element ($h = 1/32$).}  
\label{tab-conf:cond-case3}
\begin{tabular}{|l|c|ccccc|}
\hline  \hline
\multicolumn{2}{|c|}{\multirow{2}{*}{$\kappa$}}  & \multicolumn{5}{c|}{$\log _{10}(\mu)$} \\
\cline{3-7} \multicolumn{2}{|c|}{} & -4 & -2 & 0 & 2 & 4 \\
\hline 
\multirow{5}{*}{\rotatebox[origin=c]{90}{$\log_{10}(K)$}}
& 4      & 6.00 & 4.57e+02 & 4.54e+04    & 4.54e+06  & 4.54e+08 \\
& 2      & 1.45 & 6.00     & 4.57e+02    & 4.54e+04  & 4.54e+06 \\
& 0      & 1.41 & 1.45     & 6.00        & 4.57e+02  & 4.54e+04 \\
& -2     & 1.41 & 1.41     & 1.45        & 6.00      & 4.57e+02 \\
& -4     & 1.41 & 1.41     & 1.41        & 1.45      & 6.00 \\
\hline
\hline 
\multicolumn{2}{|c|}{\multirow{2}{*}{$\kappa_{\mathrm{eff}}$}}  & \multicolumn{5}{c|}{$\log _{10}(\mu)$} \\
\cline{3-7} \multicolumn{2}{|c|}{} & -4 & -2 & 0 & 2 & 4 \\
\hline 
\multirow{5}{*}{\rotatebox[origin=c]{90}{$\log_{10}(K)$}}
& 4      & 1.10(7) & 1.10(8) & 1.10(6) & 1.10(5) & 1.10(5) \\
& 2      & 1.08(6) & 1.10(6) & 1.10(5) & 1.10(5) & 1.10(5) \\
& 0      & 1.08(5) & 1.08(5) & 1.10(7) & 1.10(5) & 1.10(5) \\
& -2     & 1.08(6) & 1.08(6) & 1.08(6) & 1.10(7) & 1.10(5) \\
& -4     & 1.08(6) & 1.08(6) & 1.08(6) & 1.08(6) & 1.10(7) \\
\hline \hline
\end{tabular}
\end{table}

\begin{table}[h]
	\renewcommand{\arraystretch}{1.1}
\centering
\caption{The condition number $\kappa$, the effective condition number $\kappa_{\mathrm{eff}}$, and the iteration counts (in parentheses) for \textbf{Case EN} using $H(\operatorname{div})$-conforming element ($h = 1/32$).}  
\label{tab-conf:cond-case4}
\begin{tabular}{|l|c|ccccc|}
\hline  \hline 
\multicolumn{2}{|c|}{\multirow{2}{*}{$\kappa$}}  & \multicolumn{5}{c|}{$\log _{10}(\mu)$} \\
\cline{3-7} \multicolumn{2}{|c|}{} & -4 & -2 & 0 & 2 & 4 \\
\hline 
\multirow{5}{*}{\rotatebox[origin=c]{90}{$\log_{10}(K)$}}
& 4      & 1.14      & 1.14       & 1.14      & 1.14 & 1.14  \\
& 2      & 1.42      & 1.14       & 1.14      & 1.14 & 1.14  \\
& 0      & 2.07e+01  & 1.42       & 1.14      & 1.14 & 1.14  \\
& -2     & 1.69e+03  & 2.06e+01   & 1.42      & 1.14 & 1.14  \\
& -4     & 9.57e+04  & 9.51e+02   & 1.78e+01  & 1.42 & 1.14  \\
\hline
\hline 
\multicolumn{2}{|c|}{\multirow{2}{*}{$\kappa_{\mathrm{eff}}$}}  & \multicolumn{5}{c|}{$\log _{10}(\mu)$} \\
\cline{3-7} \multicolumn{2}{|c|}{} & -4 & -2 & 0 & 2 & 4 \\
\hline 
\multirow{5}{*}{\rotatebox[origin=c]{90}{$\log_{10}(K)$}}
& 4      & 1.09(6) & 1.08(6) & 1.08(6) & 1.08(6) & 1.08(6) \\
& 2      & 1.14(6) & 1.09(6) & 1.08(6) & 1.08(6) & 1.08(6) \\
& 0      & 1.14(7) & 1.14(6) & 1.09(6) & 1.08(6) & 1.08(6) \\
& -2     & 1.14(6) & 1.14(6) & 1.14(7) & 1.09(6) & 1.08(6) \\
& -4     & 1.14(5) & 1.14(5) & 1.14(6) & 1.14(7) & 1.09(6) \\
\hline \hline
\end{tabular}
\end{table}
\begin{table}[h]
\renewcommand{\arraystretch}{1.1}
\centering
\caption{The condition numbers and iteration counts (in parentheses) for \textbf{Case NN} and \textbf{Case EE} using nonconforming element with a mesh size of $h = 1/32$.}  
\label{tab-non:cond-case12}
\begin{tabular}{|l|c|ccccc|}
\hline  \hline
\multicolumn{2}{|c|}{\multirow{2}{*}{\textbf{NN}}}  & \multicolumn{5}{c|}{$\log _{10}(\mu)$} \\
\cline{3-7} \multicolumn{2}{|c|}{} & -4 & -2 & 0 & 2 & 4 \\
\hline 
\multirow{5}{*}{\rotatebox[origin=c]{90}{$\log_{10}(K)$}}
& 4      & 1.08(5) & 1.08(5) & 1.08(5) & 1.08(5) & 1.08(5) \\
& 2      & 1.05(5) & 1.08(5) & 1.08(5) & 1.08(5) & 1.08(5) \\
& 0      & 1.05(5) & 1.05(5) & 1.08(5) & 1.08(5) & 1.08(5) \\
& -2     & 1.05(5) & 1.05(5) & 1.05(5) & 1.08(5) & 1.08(5) \\
& -4     & 1.05(5) & 1.05(5) & 1.05(5) & 1.05(5) & 1.08(5) \\
\hline
\hline 
\multicolumn{2}{|c|}{\multirow{2}{*}{\textbf{EE}}}  & \multicolumn{5}{c|}{$\log _{10}(\mu)$} \\
\cline{3-7} \multicolumn{2}{|c|}{} & -4 & -2 & 0 & 2 & 4 \\
\hline 
\multirow{5}{*}{\rotatebox[origin=c]{90}{$\log_{10}(K)$}}
& 4      & 1.11(6) & 1.10(6) & 1.10(5) & 1.10(5) & 1.10(5) \\
& 2      & 1.27(6) & 1.11(5) & 1.10(5) & 1.10(5) & 1.10(5) \\
& 0      & 1.40(5) & 1.27(5) & 1.11(5) & 1.10(5) & 1.10(5) \\
& -2     & 1.40(5) & 1.40(5) & 1.27(5) & 1.11(5) & 1.10(5) \\
& -4     & 1.40(5) & 1.40(5) & 1.40(5) & 1.27(5) & 1.11(5) \\
\hline \hline
\end{tabular}
\end{table}

According to the theory, the condition numbers should be robust for \textbf{Case NN} and \textbf{Case EE}. As shown in \cref{tab-conf:cond-case12} and \cref{tab-non:cond-case12}, the results confirm robustness with respect to the physical parameters, aligning with these theoretical predictions.

For \textbf{Case NE} and \textbf{Case EN}, our analysis showed that while the condition number $\kappa$ depends on the physical parameters, the effective condition number $\kappa_{\mathrm{eff}}$ is independent of them. In \cref{tab-conf:cond-case3} and \cref{table:cond-case3}, we present the condition and effective condition numbers varying the physical parameters for \textbf{Case NE}.  As observed, the condition number scales as $\max\{\sqrt{\mu K}, 1\}$, consistent with \cref{lemma: NE}, while the effective condition number remains parameter-robust.  Consequently, the number of iterations remains stable, demonstrating the robustness of the proposed preconditioner and validating \cref{thm:ConditionNumber-Case3&4}. \cref{tab-conf:cond-case4} and \cref{table:cond-case4} present the condition and effective condition numbers for \textbf{Case EN}. As expected, we observe similar results to \textbf{Case NE}.  Specifically, the condition number scales as $\max\{1/\sqrt{\mu K}, 1\}$, while the effective condition number is parameter-robust, leading to stable iteration counts.  These findings are consistent with \cref{lemma: EN} and \cref{thm:ConditionNumber-Case3&4}, further validating the robustness of the preconditioner.

\begin{table}[h]
	\renewcommand{\arraystretch}{1.1}
\centering
\caption{The condition number $\kappa$, the effective condition number $\kappa_{\mathrm{eff}}$, and the iteration counts (in parentheses) for \textbf{Case NE} using nonconforming element ($h = 1/32$.)
}  
\label{table:cond-case3}
\begin{tabular}{|l|c|ccccc|}
\hline  \hline 
\multicolumn{2}{|c|}{\multirow{2}{*}{$\kappa$}}  & \multicolumn{5}{c|}{$\log _{10}(\mu)$} \\
\cline{3-7} \multicolumn{2}{|c|}{} & -4 & -2 & 0 & 2 & 4 \\
\hline 
\multirow{4}{*}{\rotatebox[origin=c]{90}{$\log_{10}(K)$}}
& 4      & 5.94 & 4.52e+02 & 4.49e+04    & 4.49e+06  & 4.50e+08 \\
& 2      & 1.45 & 5.94     & 4.52e+02    & 4.49e+04  & 4.49e+06 \\
& 0      & 1.41 & 1.45     & 5.94        & 4.52e+02  & 4.49e+04 \\
& -2     & 1.40 & 1.41     & 1.45        & 5.94      & 4.52e+02 \\
& -4     & 1.40 & 1.40     & 1.41        & 1.45      & 5.94 \\
\hline
\hline 
\multicolumn{2}{|c|}{\multirow{2}{*}{$\kappa_{\mathrm{eff}}$}}  & \multicolumn{5}{c|}{$\log _{10}(\mu)$} \\
\cline{3-7} \multicolumn{2}{|c|}{} & -4 & -2 & 0 & 2 & 4 \\
\hline 
\multirow{4}{*}{\rotatebox[origin=c]{90}{$\log_{10}(K)$}}
& 4      & 1.10(7) & 1.10(8) & 1.10(6) & 1.10(5) & 1.10(5) \\
& 2      & 1.08(6) & 1.10(7) & 1.10(5) & 1.10(5) & 1.10(5) \\
& 0      & 1.08(5) & 1.08(5) & 1.10(7) & 1.10(5) & 1.10(5) \\
& -2     & 1.08(5) & 1.08(5) & 1.08(5) & 1.10(7) & 1.10(5) \\
& -4     & 1.08(5) & 1.08(5) & 1.08(5) & 1.08(5) & 1.10(7) \\
\hline \hline
\end{tabular}
\end{table}

\begin{table}[h]
	\renewcommand{\arraystretch}{1.1}
\centering
\caption{The condition number $\kappa$, the effective condition number $\kappa_{\mathrm{eff}}$, and the iteration counts (in parentheses) for \textbf{Case EN} using nonconforming element  ($h = 1/32$).}  
\label{table:cond-case4}
\begin{tabular}{|l|c|ccccc|}
\hline  \hline
\multicolumn{2}{|c|}{\multirow{2}{*}{$\kappa$}}  & \multicolumn{5}{c|}{$\log _{10}(\mu)$} \\
\cline{3-7} \multicolumn{2}{|c|}{} & -4 & -2 & 0 & 2 & 4 \\
\hline 
\multirow{4}{*}{\rotatebox[origin=c]{90}{$\log_{10}(K)$}}
& 4      & 1.09      & 1.08       & 1.08      & 1.08 & 1.08  \\
& 2      & 1.41      & 1.09       & 1.08      & 1.08 & 1.08  \\
& 0      & 1.79e+01  & 1.41       & 1.09      & 1.08 & 1.08  \\
& -2     & 1.41e+03  & 1.79e+01   & 1.41      & 1.09 & 1.08  \\
& -4     & 1.40e+05  & 1.41e+03   & 1.79e+01  & 1.41 & 1.09  \\
\hline
\hline 
\multicolumn{2}{|c|}{\multirow{2}{*}{$\kappa_{\mathrm{eff}}$}}  & \multicolumn{5}{c|}{$\log _{10}(\mu)$} \\
\cline{3-7} \multicolumn{2}{|c|}{} & -4 & -2 & 0 & 2 & 4 \\
\hline 
\multirow{4}{*}{\rotatebox[origin=c]{90}{$\log_{10}(K)$}}
& 4      & 1.07(6) & 1.07(6) & 1.07(5) & 1.07(5) & 1.07(5) \\
& 2      & 1.05(6) & 1.07(5) & 1.07(5) & 1.07(5) & 1.07(5) \\
& 0      & 1.05(7) & 1.05(5) & 1.07(5) & 1.07(5) & 1.07(5) \\
& -2     & 1.05(8) & 1.05(7) & 1.05(5) & 1.07(5) & 1.07(5) \\
& -4     & 1.05(5) & 1.05(5) & 1.05(7) & 1.05(5) & 1.07(5) \\
\hline \hline
\end{tabular}
\end{table}

\subsection{Inexact Preconditioner}
In the previous subsection, the preconditioner $\mathcal{B}$ \eqref{preconditioner} is implemented using a direct solver.  While this preconditioner is parameter-robust, it can still be expensive in practice.  Specifically, the block $I_p^{-1}$ only requires inverting a scaled mass matrix associated with the pressure finite-element spaces, which is generally straightforward.   For example, we use $\mathbb{P}_0$ for the pressures, therefore, we only need to inverted a diagonal matrix in our implementation.  On the other hand, inverting the block $A_{\mathbf{u}} + D_{\mathbf{u}}$ could be quite challenging due to the presence of the $(\operatorname{div},\operatorname{div})$-like term, i.e., $D_{\mathbf{u}}$.  This term is well-known for introducing a large, nontrivial near-kernel space for $A_{\mathbf{u}}+D_{\mathbf{u}}$, making it difficult to handle computationally. 

In this subsection, we adopt unsmoothed aggregation Algebraic Multigrid Method (UA-AMG) \cite{wang2013parallel,hu2019adaptive} to solve the $A_{\mathbf{u}}+D_{\mathbf{u}}$ block inexactly.  In addition, following \cite{arnold2000multigrid,adler2022monolithic}, we apply an overlapping block smoother on the finest level to handle the large near-kernel introduced by the $(\operatorname{div},\operatorname{div})$-like term $D_{\mathbf{u}}$.  Specifically, we use vertex patch as the blocks in order to suitably handle the locally-supported basis functions for the divergence-free space.  Finally, the block $A_{\mathbf{u}}+D_{\mathbf{u}}$ is solved by the GMRes method with UA-AMG as a preconditioner. To ensure robustness, the tolerance for this approximate solve is set to $10^{-2}$.   Because the inner Krylov iterations introduce variability, we employ flexible GMRes as the outer Krylov iterative method here.

\begin{table}[h]
\centering
\caption{The iteration counts of the inexact solver in different cases using $H(\operatorname{div})$-conforming element with a mesh size of $h = 1/32$.}  
\label{tab-inexact:iteration number}
\begin{minipage}{0.45\textwidth}
\begin{tabular}{|l|r|rrrrr|}
\hline  \hline 
\multicolumn{2}{|c|}{\multirow{2}{*}{\textbf{NN}}}  & \multicolumn{5}{c|}{$\log _{10}(\mu)$} \\
\cline{3-7} \multicolumn{2}{|c|}{} & -4 & -2 & 0 & 2 & 4 \\
\hline 
\multirow{5}{*}{\rotatebox[origin=c]{90}{$\log_{10}(K)$}}
& 4      & 13 & 16 & 15 & 12 & 12 \\
& 2      & 9  & 10 & 12 & 12 & 12 \\
& 0      & 9  & 10 & 12 & 12 & 12 \\
& -2     & 13 & 10 & 10 & 12 & 11 \\
& -4     & 14 & 11 & 10 & 10 & 10 \\
\hline
\hline 
\multicolumn{2}{|c|}{\multirow{2}{*}{\textbf{NE}}}  & \multicolumn{5}{c|}{$\log _{10}(\mu)$} \\
\cline{3-7} \multicolumn{2}{|c|}{} & -4 & -2 & 0 & 2 & 4 \\
\hline 
\multirow{5}{*}{\rotatebox[origin=c]{90}{$\log_{10}(K)$}}
& 4      & 17 & 29 & 17 & 12 & 12 \\
& 2      & 11 & 13 & 12 & 12 & 12 \\
& 0      & 11 & 11 & 13 & 12 & 12 \\
& -2     & 19 & 13 & 13 & 13 & 11 \\
& -4     & 24 & 16 & 13 & 13 & 13 \\
\hline \hline 
\end{tabular}
\end{minipage}
\begin{minipage}{0.45\textwidth}
\begin{tabular}{|l|r|rrrrr|}
\hline  \hline
\multicolumn{2}{|c|}{\multirow{2}{*}{\textbf{EE}}}  & \multicolumn{5}{c|}{$\log _{10}(\mu)$} \\
\cline{3-7} \multicolumn{2}{|c|}{} & -4 & -2 & 0 & 2 & 4 \\
\hline 
\multirow{5}{*}{\rotatebox[origin=c]{90}{$\log_{10}(K)$}}
& 4      & 18 & 22 & 16 & 13 & 11 \\
& 2      & 11 & 11 & 11 & 11 & 11 \\
& 0      & 12 & 11 & 11 & 11 & 11 \\
& -2     & 17 & 14 & 17 & 15 & 11 \\
& -4     & 25 & 16 & 17 & 15 & 14 \\
\hline
\hline 
\multicolumn{2}{|c|}{\multirow{2}{*}{\textbf{EN}}}  & \multicolumn{5}{c|}{$\log _{10}(\mu)$} \\
\cline{3-7} \multicolumn{2}{|c|}{} & -4 & -2 & 0 & 2 & 4 \\
\hline 
\multirow{5}{*}{\rotatebox[origin=c]{90}{$\log_{10}(K)$}}
& 4      & 14 & 17 & 17 & 14 & 11 \\
& 2      & 11 & 11 & 11 & 11 & 11 \\
& 0      & 13 & 12 & 11 & 11 & 11 \\
& -2     & 14 & 14 & 12 & 11 & 11 \\
& -4     & 15 & 12 & 14 & 12 & 11 \\
\hline \hline
\end{tabular}
\end{minipage}
\end{table}

\begin{table}[h]
\centering
\caption{The iteration counts of the inexact solver in different cases using nonconforming element with a mesh size of $h = 1/32$.}  
\label{tab-inexact:iteration number-non}
\begin{minipage}{0.45\textwidth}
\begin{tabular}{|l|r|rrrrr|}
\hline  \hline 
\multicolumn{2}{|c|}{\multirow{2}{*}{ \textbf{NN}}}  & \multicolumn{5}{c|}{$\log _{10}(\mu)$} \\
\cline{3-7} \multicolumn{2}{|c|}{} & -4 & -2 & 0 & 2 & 4 \\
\hline 
\multirow{5}{*}{\rotatebox[origin=c]{90}{$\log_{10}(K)$}}
& 4      & 11 & 16 & 16 & 11 & 9 \\
& 2      & 8  & 10 & 10 & 9  & 9 \\
& 0      & 8  & 10 & 10 & 9  & 9 \\
& -2     & 10 & 10 & 10 & 10 & 9 \\
& -4     & 10 & 10 & 10 & 10 & 10 \\
\hline
\hline 
\multicolumn{2}{|c|}{\multirow{2}{*}{\textbf{NE}}}  & \multicolumn{5}{c|}{$\log _{10}(\mu)$} \\
\cline{3-7} \multicolumn{2}{|c|}{} & -4 & -2 & 0 & 2 & 4 \\
\hline 
\multirow{4}{*}{\rotatebox[origin=c]{90}{$\log_{10}(K)$}}
& 4      & 17 & 31 & 18 & 13 & 10 \\
& 2      & 10 & 12 & 10 & 10 & 9 \\
& 0      & 11 & 11 & 11 & 9  & 9 \\
& -2     & 11 & 11 & 11 & 11 & 9 \\
& -4     & 11 & 11 & 11 & 11 & 11 \\
\hline \hline 
\end{tabular}
\end{minipage}
\begin{minipage}{0.45\textwidth}
\begin{tabular}{|l|r|rrrrr|}
\hline  \hline 
\multicolumn{2}{|c|}{\multirow{2}{*}{\textbf{EE}}}  & \multicolumn{5}{c|}{$\log _{10}(\mu)$} \\
\cline{3-7} \multicolumn{2}{|c|}{} & -4 & -2 & 0 & 2 & 4 \\
\hline 
\multirow{5}{*}{\rotatebox[origin=c]{90}{$\log_{10}(K)$}}
& 4      & 17 & 28 & 17 & 13 & 10 \\
& 2      & 10 & 11 & 10 & 10 & 10 \\
& 0      & 11 & 11 & 10 & 10 & 10 \\
& -2     & 11 & 11 & 11 & 10 & 10 \\
& -4     & 11 & 11 & 11 & 11 & 10 \\
\hline
\hline 
\multicolumn{2}{|c|}{\multirow{2}{*}{\textbf{EN}}}  & \multicolumn{5}{c|}{$\log _{10}(\mu)$} \\
\cline{3-7} \multicolumn{2}{|c|}{} & -4 & -2 & 0 & 2 & 4 \\
\hline 
\multirow{5}{*}{\rotatebox[origin=c]{90}{$\log_{10}(K)$}}
& 4      & 14 & 19 & 19 & 12 & 10 \\
& 2      & 9  & 10 & 10 & 10 & 10 \\
& 0      & 11 & 11 & 10 & 10 & 10 \\
& -2     & 15 & 12 & 11 & 10 & 10 \\
& -4     & 10 & 10 & 11 & 11 & 10 \\
\hline \hline 
\end{tabular}
\end{minipage}
\end{table}

\cref{tab-inexact:iteration number}-\cref{tab-inexact:iteration number-non} present the iteration counts for the inexact block preconditioner. The results show that the iteration counts remain robust with respect to the physical parameters for both discretizations. This indicates that the UA-AMG method, combined with a properly chosen block smoother, has the potential to provide an effective inexact version of the block preconditioner $\mathcal{B}$ \eqref{preconditioner}. However, the development and theoretical analysis of a fully inexact and scalable block preconditioner remain subjects of ongoing research and will be addressed in future work.

\section{Conclusions}
In this paper, we studied the Stokes-Darcy coupled problem, focusing on developing a parameter-robust preconditioner that avoids fractional operators.  By imposing the normal flux continuity interface condition directly within the finite-element spaces, we established  well-posedness of coupled Stokes-Darcy systems equipped under various boundary conditions.  Specifically, the inf-sup constants are parameter robust for \textbf{Case NN} and \textbf{Case EE}, while they depend on the physical parameters $\mu$ and $K$ for \textbf{Case NE} and \textbf{Case EN}.  However, parameter robustness can be restored for these cases by considering quotient spaces. 

Building on the well-posedness analysis and leveraging the operator preconditioning framework, we developed and analyzed a block diagonal preconditioner that is robust with respect to the physical and discretization parameters. Notably, by enforcing the normal flux continuity conditions on the interface in the definition of the finite-element spaces, instead of using Lagrange multipliers, our analysis does not employ fractional Sobolev norms. In turn, our block preconditioner does not involve any fractional operators, enhancing its practicality for implementation.  Numerical results confirmed the theoretical results and demonstrated the robustness of the precondtioners.  

Furthermore, we explore the development of an inexact version of the preconditioner using the UA-AMG method with a block smoother to address the large near-kernel introduced by the $(\operatorname{div},\operatorname{div})$-like term in the velocity block.  Preliminary results highlighted the potential effectiveness of this approach.   Further theoretical and numerical investigation into this inexact preconditioning strategy remains a subject of our ongoing research.

\bibliographystyle{siamplain}
\bibliography{references}
\end{document}

%% file: ex_shared.tex
\usepackage{lipsum}
\usepackage{amsfonts}
\usepackage{graphicx}
\usepackage{epstopdf}
\usepackage{algorithmic}
\ifpdf
  \DeclareGraphicsExtensions{.eps,.pdf,.png,.jpg}
\else
  \DeclareGraphicsExtensions{.eps}
\fi

\newsiamremark{remark}{Remark}
\newsiamremark{hypothesis}{Hypothesis}
\crefname{hypothesis}{Hypothesis}{Hypotheses}
\newsiamthm{claim}{Claim}

\headers{Parameter-robust preconditioners for Stokes-Darcy}{Xue Wang, Xiaozhe Hu, Wietse M. Boon}

\title{Parameter-robust Preconditioners for the Stokes-Darcy Coupled Problem without Fractional Operators \thanks{Submitted to the editors DATE.
\funding{The work of Hu is partially supported by the National Science Foundation under grant DMS-220826.}}}

\author{ Wietse M. Boon\thanks{NORCE Norwegian Research Centre, N-5838 Bergen, Norway. \email{wibo@norceresearch.no}.}
   \and Xiaozhe Hu\thanks{ Department of Mathematics, Tufts University, Medford, MA 02155, USA. \email{xiaozhe.hu@tufts.edu}.}
   \and Xue Wang\thanks{School of Mathematics, Shandong University, Jinan, 250100, China. \email{wangxsdu@mail.sdu.edu.cn.}}
   }

\usepackage{amsopn}
